\newcommand{\Rb}[1]{{\mathbb{R}_{#1}}}
\newcommand{\DEQS}{\begin{eqnarray*}}
\newcommand{\EEQS}{\end{eqnarray*}}
\newcommand{\DEQSZ}{\begin{eqnarray}}
\newcommand{\EEQSZ}{\end{eqnarray}}

\newcommand{\NN}{{\mathbb{N}}}

\newcommand{\lk}{\left}
\newcommand{\lqq}{\lefteqn}
\newcommand{\rk}{\right}
\newcommand{\la}{{\langle}}
\newcommand{\ra}{{\rangle}}

\newcommand{\CU}{{{ \mathcal U }}}

\newcommand{\ep}{{\epsilon }}

\newcommand{\cl}{{{ \mathfrak{l} }}}
\newcommand{\ce}{{{ \mathfrak{e} }}}

\newcommand{\CA}{{{ \mathcal A }}}

\newcommand{\CS}{{{ \mathcal S }}}

\newcommand{\CB}{{{ \mathcal B }}}
\newcommand{\EE}{\mathbb{E}}
\newcommand{\PP}{\mathbb{P}}
\newcommand{\CM}{{{ \mathcal M }}}

\newcommand{\CF}{{{ \mathcal F }}}

\newcommand{\RR}{{\mathbb{R}}}

\newcommand{\bNN}{{\bar{\mathbb{N}}}}

\newcommand\del[1]{}
\newcommand\think[1]{}
\newcommand\new[1]{}
\newcommand\zus[1]{}

\newcommand\comd[1]{} 
\newcommand\Redd[1]{} 

\newcounter{g11}
\newenvironment{list-alpha}
{\begin{list} { (\alph{g11}):}
{\usecounter{g11}
\setlength{\leftmargin}{0.7cm}
\setlength{\topsep}{0.1cm}
\setlength{\itemsep}{0.0cm}
\setlength{\parsep}{0.1cm}
\setlength{\itemindent}{0.4cm}
\setlength{\parskip}{0.0cm}}}
{\end{list}}
\newcounter{gg1}

\newcounter{lil}

\documentclass[11pt,a4,fleqn,english]{amsart}

\usepackage{amsmath}   
\usepackage{amsfonts,amssymb}

\usepackage{graphicx} 
\usepackage{url}      

\usepackage{amsfonts}
\usepackage{latexsym}
\input colordvi

\numberwithin{equation}{section}

\begin{document}

\title[ Inequalities of the Ito stochastic integral  \today]{
Burkholder-Davis-Gundy type  Inequalities of the It\^o stochastic integral  with respect to L\'evy noise
on Banach spaces}
\thanks{This work was supported by the FWF-Project P17273-N12. I would also like to thank Brze{\'z}niak for discussions on this topic.}

\author{Erika Hausenblas}
\address{Department of Mathematics, University of Salzburg,
Hellbrunnerstr. 34, 5020 Salzburg, Austria} \email{
erika.hausenblas@sbg.ac.at}

\date{\today}

\maketitle

\begin{abstract}

The aim of this note is  to give some Burkholder-Davis-Gundy type  inequalities 
which are valid for the Ito stochastic integral with respect to Banach valued L\'evy noise.

\end{abstract}



\newtheorem{theorem}{Theorem}[section]
\newtheorem{assumption}{Assumption}[section]
\newtheorem{claim}{Claim}[section]
\newtheorem{lemma}[theorem]{Lemma}
\newtheorem{corollary}[theorem]{Corollary}
\newtheorem{example}[theorem]{Example}
\newtheorem{tlemma}{Technical Lemma}[section]
\newtheorem{definition}[theorem]{Definition}
\newtheorem{remark}[theorem]{Remark}
\newtheorem{proposition}[theorem]{Proposition}
\newtheorem{Notation}{Notation}
\renewcommand{\theNotation}{}

\renewcommand{\labelenumi}{\alph{enumi}.)}

%


%
\textbf{Keywords and phrases:} {Stochastic integral of jump type,
Poisson random measures, L\'evy process.}

\textbf{AMS subject classification (2002):} {Primary 60H15;
Secondary 60G57.}

\section{Introduction}

Let us assume that  $(S,\CS)$ {is a metric space with Borel $\sigma$ algebra $\CS$} and
$\tilde
\eta$ is a  time homogeneous compensated Poisson random measure
defined on  a filtered probability space $(\Omega;\CF;(\CF_t)_{0\le t<\infty};\PP)$
with intensity measure $\nu$ on $S$, to be specified later. Let us assume that
 $1< p\le 2$,  $E$ is a Banach space of martingale type $p$, see e.g.\ the Appendix of \cite{maxreg} for a definition.
We consider in the following the It\^o stochastic integral $I=\{ I(t), 0\le t<\infty\}$ driven by the compensated 
Poisson random measure $\tilde \eta$, i.e.\ 
\DEQS 
I(t   ) &=& \int_0^t \int_{ S}\xi(s;
                  x) \tilde \eta(dx;ds),
\EEQS
where
$\xi:[0,T]\times \Omega\times S \to E$ is
a progressively measurable  processes satisfying certain integrability conditions specified later. 
We are interested in Inequalities 
satisfied by the process $I$.
In particular, we will show  that for any $q=p ^ n$, $n\in\NN$, there exist constants $C$ and $\bar C$,
only depending on $E$, $p$ and $q$,  such that 
\DEQS
\EE \sup_{0\le s\le t} \lk|I(s)\rk|^q \le C\, \EE \lk( \int_0^t \int_{ S}\lk|\xi(s;
                  x)  \rk| ^ p  \eta(dx;ds)\rk)^\frac qp 
                  \\
                  \le \bar C\,\lk( \EE  \int_0^t \int_{ S}\lk|\xi(s;
                  x)  \rk|^q \nu(dx;ds) +   \EE \lk( \int_0^t \int_{ S}\lk|\xi(s;
                  x)  \rk| ^ p \nu(dx;ds) \rk) ^\frac qp\rk).
\EEQS
From this inequalities one can derive similar inequalities for martingales 
of pure jump type. To be more precise, let $X$ be a martingale, such that there exists a L\'evy process $L$ and
a progressively process $h:[0,\infty)\to  L(E,E)$, satisfying some integrabilities condition specified later, with
$$
X(t) = \int_0 ^ t h(s) \, dL(s),\quad t\ge 0.
$$
Then there exist constants $C>0$ and $\bar C>0$ such that 
\DEQS
\lqq{
\EE \sup_{0\le s\le t} \lk|X(s)\rk|^q \le  C\, \EE \lk( \sum_{0\le s\le t} \lk|\Delta_s X   \rk| ^ p \rk) ^\frac qp}
&&
\\ 
                                  & \le& \bar C\,\lk( \EE \sum_{0\le s\le t} \lk|\Delta_s X   \rk| ^q  +
   \EE \lk( \sum_{0\le s\le t}\EE[ \, \lk|\Delta_s X   \rk|  ^ p|\CF_{s-}]  \rk) ^\frac qp\rk),
\EEQS
where $\Delta_t X = X(t)-X(t-)$, $t>0$.

\begin{Notation}
Let $\NN_0:=\NN\cup\{0\}$ and $\bar\NN := \NN_0\cup\{\infty\}$.
By $\CM_I( S\times \mathbb{R}_+)$ we denote the family of all $\bar{\mathbb{N}}$-valued measures on
$(S\times \mathbb{R}_+,\CS\otimes \mathcal{B}({\mathbb{R}_+}))$.\del{ and
$\CM_{S\times \mathbb{R}_+} ^ {\bar{\mathbb{N}}}$ is the $\sigma$-field on $M_{S\times \RR_+} ^{ \bNN}$ generated by functions
$i_B:M\ni\mu \mapsto \mu(B)\in \bNN$, $B\in \CS\otimes \mathcal{B}({\mathbb{R}_+})$.}
By $\CM^+(S)$ we denote the set of all positive measures on $S$.
For any Banach space $Y$ and number $q\in[1,\infty)$, we denote by  $\mathcal{N}(\RR_+;Y)$ 
 the space of (equivalence classes) of
{progressively-measurable}  processes $\xi :\RR_+\times\Omega\to Y$ and by 
$\mathcal{M}^p (\RR_+;Y)$  the Banach space consisting of those $\xi \in \mathcal{N}(\RR_+;Y)$ for which
$
\mathbb{E}\int_0^\infty  \vert\xi (t)\vert^p_Y\,dt<\infty$.
\end{Notation}

\section{Main results}\label{sec_main}

Let us first introduce the notation of time homogeneous Poisson random measures over a filtered probability space.
\begin{definition}\label{def-Prm} 

Let $(S,\CS)$ be a measurable space and let $(\Omega,\CF,\PP)$ be a
probability space.\\ 
\noindent
A {\sl Poisson random measure} $\eta$
on $(S,\CS)$   over $(\Omega,\CF,\PP)$,  is a measurable function $\eta: (\Omega,\CF)\to ( \CM_I(S\times \RR_+),\CB(\CM_I(S\times \RR_+))$, where $\CB(\CM_I(S\times \RR_+)$ is the $\sigma$-field on $\CM_I(S\times \RR_+)$ generated by functions
$i_B:\CM_I(S\times \RR_+) \ni\mu \mapsto \mu(B)\in \bar{\mathbb{N}}$, $B\in \CS$, such that
\begin{trivlist}
\item[(i)] $\eta$ is independently scattered, i.e.\ if the sets
$ B_j \in   \CS\times \CB(\RR_+)$, $j=1,\cdots, n$ are pairwise disjoint,   then the random
variables $\eta(B_j)$, $j=1,\cdots,n $ are pairwise independent.
\item[(ii)] for each $B\in  \CS $,
 $\eta(B):=i_B\circ \eta : \Omega\to \bar{\mathbb{N}} $ 
 is a Poisson random variable with parameter\footnote{If  $\EE \eta(B) = \infty$, then obviously $\eta(B)=\infty$ a.s..} $\EE\eta(B)$.
\item[(iii)] for each $U\in \CS$, the $\bar{\mathbb{N}}$-valued processes $(N(t,U))_{t>0}$  defined by
$$N(t,U):= \eta((0,t]\times U), \;\; t>0$$
is $(\mathcal{F}_t)_{t\geq 0}$-adapted and its increments are independent of the past, i.e. if $t>s\geq 0$, then
$N(t,U)-N(s,U)=\eta((s,t]\times U)$ is independent of
$\mathcal{F}_s$.
\end{trivlist}
\end{definition}

\del{\begin{definition}
Let $(\Omega,\CF,(\CF_t)_{t\in\RR_+},\PP)$ be a filtered probability space and $\eta$ be a time homogeneous Poisson random measure 
over $(\Omega,\CF,\PP)$. We say that the filtration $(\CF_t)_{t\in\RR_+}$
is not anticipated to $\eta$, if
for any $t\in\RR_+$, $\CF_t$ is independent of $\CU_t := \sigma\lk\{ \eta(\cdot,(t,r]\cap \cdot),t<r<\infty\rk\}$.
If $\eta$ is adapted to $(\CF_t)_{t\in\RR_+}$ and $(\CF_t)_{t\in\RR_+}$ is not anticipated to $\eta$, we call
$\eta$ a time homogeneous Poisson random measure 
over $(\Omega,\CF,(\CF_t)_{t\in\RR_+},\PP)$.
\end{definition}
}

\noindent
Poisson random measures arise in a natural way. E.g.\
by means of a L\'evy process one can construct a Poisson random measure. 
\begin{definition} Let $E$ be a Banach space.
A stochastic process $\{ X(t):t\ge 0\}$ is a L\'evy process if the following conditions are satisfied.
\begin{itemize}

\item for any choice $n\in\NN$ and $0\le t_0<t_1<\cdots t_n$, the random variables
$X(t_0)$, $X(t_1)-X(t_0)$, $\ldots$, $X(t_n)-X(t_{n-1})$ are independent;
\item $X_0=0$ a.s.;
\item For all $0\le s<t$, the distribution of $X(t+s)-X(s)$ does not depend on $s$;
\item $X$ is stochastically continuous;
\item the trajectories of $X$ are a.s. c\'adl\'ag on $E$. 
\end{itemize}
\end{definition}

The characteristic function of a L\'evy process is uniquely defined and is given by the L\'evy-Khinchin formula.
 formula.
\del{(see arujo, Linde Theorem 5.7.3, p. 90)}
Here, we assume for simplicity that $E$ is a Hilbert space with inner product $\la\cdot,\cdot\ra$. For discussion on Banach spaces we refer e.g.\ to 
\cite{Gine,araujo,Linde}.
 Then,
for any $E$-valued L\'evy process  $\{ L(t):t\ge 0\}$  there exists a trace class operator $Q$, a non negative measure 
$\nu$ concentrated at 
$E\setminus \{0\}$ and an element $m\in E$ such that 
\DEQS
 \lqq{ \EE e^{i\la \theta ,L(t)\ra}= \exp\Big( i\la m,x\ra +\frac 12 \la Qx,x\ra }&&
\\&&\phantom{mmmmmmm}{}+
\int_E \lk( 1-e ^ {i\la x,y\ra }+1_{(-1,1)}(|y|_E)i\la x,y\ra \rk)\nu(dy)  \Big).
\EEQS
We call the measure $\nu$ {\em characteristic measure} of the L\'evy process $\{ L(t):t\ge 0\}$.
Moreover, note that the triplet $(Q,m,\nu)$ uniquely determines the law of the  L\'evy 
process. Now, one can construct a Poisson random measure with intensity measure $\nu$.

\begin{example}\label{levy}
Given a filtered probability space $(\Omega,\CF,(\CF_t)_{t\ge 0}, \PP)$
and a Hilbert space $E$.
To each 
time homogeneous $E$-valued L\'evy process   $\{ L(t):t\ge 0\}$ 
on $(E,\CB(E))$ with characteristic measure $\nu$, we can associate a counting measure, denoted by $\eta_L$ 
over $(\Omega,\CF,(\CF_t)_{t\ge 0}, \PP)$ by  
$$
\CB(E)\times \CB(\RR_+)\ni (B,I) \mapsto \eta_L(B\times I ) := \#\{s\in I \mid 
L_{s}-L_{s-}\in B\}\in\NN_0\cup\{\infty\}.
$$
The counting measure is a time homogeneous Poisson random measure with intensity measure $\nu$.
Moreover, 
$$
L(t) = \int_0 ^ t z\, \eta_L(dz,ds), \quad t\ge 0.
$$
\end{example}
For more details on the relationship between Poisson random measure and L\'evy processes we refer to Applebaum \cite{apple}
Ikeda and Watanabe \cite{ikeda} or Peszat and
Zabczyk \cite{zab}.
\\[0.2cm]
Let us assume  that
 $1< p\le 2$ and $E$ is a Banach space of martingale type $p$,  see e.g.\ the Appendix of \cite{maxreg} for a definition.
Let us assume that $(S,\CS)$ is a measurable space and
 $\nu\in M ^ +(S)$.  Suppose that $ (\Omega,\CF,(\mathcal{F}_t)_{t\geq 0},\PP)$ is a filtered
probability space, $\eta: \CS\times \CB(\RR_+)\to \bar{\mathbb{N}}$ is a time homogeneous Poisson random measure
with intensity measure $\nu$ defined over $(\Omega,\CF,(\mathcal{F}_t)_{t\geq 0},\PP)$.
We will denote by $\tilde\eta=\eta-\gamma$ the   compensated Poisson random measure  associated to $\eta$,
where the compensator $\gamma$ is defined  by
$$
\CB(\RR_+)\times \CS\ni (A,I)\mapsto  \gamma(A,I)=\nu(A)\lambda(I)\in\RR _ +.
$$

We have proved in \cite{maxreg} that there exists a unique
continuous linear operator which associates   with   each progressively measurable
process $\xi\in\mathcal{M}^p (\RR_+;L ^ p(S,\nu;E) )$ \del{$:\Rb{+}\times S\times \Omega \to E$ such that
\begin{equation}
\label{cond-2.01}\mathbb{E} \int_0^T\int_S\vert \xi(r,x)\vert^p\,
\nu(dx)\,dr <\infty, \quad T>0,
\end{equation}
}
an adapted c\'adl\'ag $E$-valued process, denoted by $\int_0^t \int_S \xi(r,x)
\tilde{\eta}(dx,dr)$, $t\geq 0$, such that if a process $\xi\in \CM(\RR_+,L ^ p(S,\nu;E))$ 
 is a
random step process with representation
$$
\xi(r) = \sum_{j=1} ^n 1_{(t_{j-1}, t_{j}]}(r)\, \xi_j,\quad r\ge 0,
$$
where $\{t_0=0<t_1<\ldots<t_n<\infty\}$ is a finite partition of $[0,\infty)$ and
for all $j$,    $\xi_j$ is  an $E$-valued $\CF_{t_{j-1}}$ measurable, $p$-summable  random variable, then
\begin{equation}
\label{eqn-2.02} \int_0^t \int_S \xi(r,x)\,
\tilde{\eta}(dx,dr)=\sum_{j=1}^n  \int_S  \xi_j (x)\,\tilde \eta \lk(dx,(t_{j-1}\wedge t, t_{j}\wedge t] \rk).
\end{equation}

The continuity mentioned above means that there exists a constant
$C=C(E)$ independent of $\xi$ such that

\begin{equation}
\label{ineq-2.03} \mathbb{E} \vert \int_0^t \int_S \xi(r,x)\,
\tilde{\eta}(dx,dr)\vert ^p \leq C\,\mathbb{E} \int_0^t\int_S\vert
\xi(r,x)\vert^p\, \nu(dx)\,dr, \; t\geq 0.
\end{equation}

As mentioned above, we are interested in inequalities satisfied by the 
stochastic integral
processes given by 
$$
\RR_+\ni t\mapsto I(t) = \int_0 ^ t \int_S  \xi(s,z)\, \tilde \eta(dz,ds).
$$
\begin{proposition} \label{integral-1}
Let $1<p\le 2$ and let $E$ be a separable Banach space of martingale type $p$.
Let $(\Omega,\CF,(\CF _t)_{t\ge 0},\PP)$  be a filtered probability space.
Assume that $\tilde \eta$ is a compensated time homogeneous Poisson random measure over $(\Omega,\CF,(\CF _t)_{t\ge 0},\PP)$ 
with intensity $\nu $ and compensator $\gamma$.
Assume that $\xi\in\mathcal{M}^p (\RR_+;L ^ p(S,\nu;E) )$. \del{ is a progressively measurable process over $(\Omega,\CF,(\CF _t)_{t\ge 0},\PP)$ taking values in $L^p(S,\nu,E)$ such that 
$$\int_0^\infty \int_S |\xi (s,z) |  ^ p\, \nu(dz)\, ds <\infty.
$$
}
\del{Let 
\DEQS
[0,\infty ) \ni t\mapsto 
\int_0 ^ t 
\int_S 
  \xi(s,z)\, \tilde  \eta(dz,ds), \quad t\ge 0.
\EEQS
}
Then 
\begin{trivlist}
%
%
%

\item[(i)]
there exists a constant $C=C_p(E)\, 2^{2-p}$ only depending on   $E$ and  $p$  such that 
\DEQS 
\lefteqn{
 \EE\sup_{0<t\le T} \lk|\int_{0} ^{t}\int_S  \xi(\sigma,z)\; \tilde \eta(dz;d\sigma)\rk| ^r
 \le
}
&&
\\
\nonumber
&&
\quad\quad C \:\lk(\int_0 ^T  \int_S
\EE \lk|\xi(s,z)\rk| ^p \;\nu(dz)\; ds\rk) ^{r\over p},\quad 0<r\le p.
\EEQS
\item[(ii)] 
there exists a constant $C= C_p(E)\, 2^{r(2+\frac {1 }p)} (m_0-1)^ {(p-1)\frac rp}$, $m_0=\inf\{n\ge 1:p-n\frac pr \le 1\}$, only depending on   $E$, $p$, and $r$ such that 
\DEQS \lefteqn{
 \EE\sup_{0<t\le T} \lk|\int_{0} ^{t}\int_S  \xi(\sigma,z)\;  \tilde \eta(dz;d\sigma)\rk| ^r \le
}
&&
\\\nonumber
&&
\quad\quad   C\, \EE\lk(\int_0 ^T  \int_S
 \lk|\xi(s,z)\rk| ^p \;\eta(dz;ds)\rk) ^{r\over p},\quad p\le r<\infty,
\EEQS

\item[(iii)] Let $q$ be a natural number with $q=p^n$ for a number $n\in\NN$.
If in addition
$$ \int_0 ^t \int_S \EE| \xi(s,z)| ^{q}\;\nu(dz)\;
ds  <\infty,
$$
then 
\DEQS
\lefteqn{
\EE\sup_{0<s\le t}\lk| \int_{0 } ^t\int_S \xi(s,z) \tilde \eta(dz;ds)\rk| ^q\le}
\\
&& \quad\quad 2^{2-p}\; \sum_{l=1} ^n \bar C(l) \EE \lk( \int_0 ^t \int_S | \xi(s,z)| ^{p ^l}\;\nu(dz)\;
ds \rk) ^{p ^{n-l}} .
\EEQS
where $\bar C(0)=1$ and $\bar C(i)= \bar C(i-1)2 ^ { p ^ {n-i+1}+2p ^ {n-1}} \,(m(n-i)-1)^ {(p-1)\frac rp} $,
$m(i)=[ p ^ {i-1}(p-1)]+1$.

\end{trivlist}
\end{proposition}
\begin{remark}
If the underlying Banach space $E$ is a Hilbert space or $\RR^d$ equipped by the Euclidean norm, then $E$ is a 
Banach space of $M$-type $2$, and $C_2(E)=1$.
For other cases we refer to the book of Linde 
\cite{Linde}.

\end{remark}
%
Assume for the next paragraph, that $E$ is a Hilbert space and 
that the time homogeneous Poisson random measure is the counting measure
of the L\'evy process described in Example \ref{levy}. 
But before we look at the formulation of Proposition \ref{integral-1}
in terms of L\'evy processes, we introduce a important process associated to a L\'evy process.
The {\em jump process} $\Delta X = \{\Delta _tX,\,0\le t<\infty\}$
of a process $X$ is given by 
$$
\Delta_t X(t) := X(t) - X(t-)= X(t) - \lim_{\ep\to 0} X(t-\ep) ,\quad t\ge 0.
$$
Assume that $X$ arises by stochastic integration of a L\'evy process of pure jump type.
In particular, we assume that there exists a L\'evy process $L$ and a  c\'adl\'ag process $h\in \CM^p(\RR_+, L(Z,E))$
such that
$$
X(t)  := \int_0 ^ t h(s)\, dL(s),\quad t\ge 0.
$$
Then, $\Delta_tX= h(t)\Delta_t L$, $t\ge 0$.
Now, the Proposition \ref{integral-1} reads 
as follows. 

\begin{corollary}
Let $(\Omega,\CF,(\CF_t)_{t\ge 0}, \PP)$ be  a filtered probability space
and $E$ is a Hilbert space.
Let  $L=\{ L(t), \, 0\le t<\infty\} $ be 
time homogeneous $E$-valued L\'evy process   
with characteristic measure $\nu$
over $(\Omega,\CF,(\CF_t)_{t\ge 0}, \PP)$,
let $h\in \CM^p(\RR_+, L(E,\nu;E))$ be a c\'adl\'ag process such that $h:[0,\infty)\to L(E,E)$ and
$X=\{ X(t), \, 0\le t<\infty\}$ be given by 
$$X(t):= \int_0 ^ t h(s)\, dL(s),\quad t\ge 0.
$$
Then
\begin{trivlist}
\item[(ii)] 
there exists a constant $C= C_p(E)\, 2^{r(2+\frac {1 }p)} (m_0-1)^ {(p-1)\frac rp}$, $m_0=\inf\{n\ge 1:p-n\frac pr \le 1\}$,  only depending on   $E$, $p$, and $r$ such that 
\DEQS 
 \EE\sup_{0<t\le T} \lk| X(t) \rk| ^r \le
  C\, \EE\lk(\sum_{s\le t}  \lk| \Delta _sX \rk| ^p\rk) ^{r\over p},\quad p\le r<\infty,
\EEQS

\item[(iii)] Let $q$ be a natural number with $q=p^n$ for a number $n\in\NN$.
If in addition
$$ \EE \sum_{s\le t}  \lk| \Delta _sX \rk| ^q \;
 <\infty,
$$
then 
\DEQS
\EE\sup_{0<s\le t}\lk| X(t)\rk| ^q &\le& 2^{2-p}\; \sum_{l=1} ^n \bar C(l) \EE \lk( \sum_{s\le t} \EE\lk[
 \lk| \Delta _sX \rk|  ^{p ^l}\mid \CF_{s-}\rk]\;
 \rk) ^{p ^{n-l}} .
\EEQS
where $\bar C(0)=1$ and $\bar C(i)= \bar C(i-1)2 ^ { p ^ {n-i+1}+2p ^ {n-1}} (m(n-i)-1)^ {(p-1)\frac rp} $,
$m(i)=[p ^ {i-1}(p-1)]+1$.
\end{trivlist}
\end{corollary}

\section{Proof of the Inequalities in Proposition \ref{integral-1}}
The proof of Inequality (i) is taken from Corollary C.2 of Brze{\'z}niak and Hausenblas 
\cite{maxreg}. If $r= p$ Inequality (ii) follows by the definition of the compensator.
Hence, we give here a proof valid for $r\ge p$.
Inequality (iii) can be shown by induction on $n$ and can be found also in \cite{MR98c:60063} or \cite{MR88b:60113}.

\begin{proof}[Proof of Inequality (i):] 
Before beginning let us state the following Lemma.
The proof of this Lemma can be found by direct calculation or in \cite[Lemma C.2]{maxreg}.
\begin{lemma}
\label{lem-cw6}
Suppose that $\xi \sim \text{ Poiss }(\lambda)$,
where $\lambda >0$. Then, for all $p\in [1,2]$,
$$\mathbb{E}|\xi-\lambda|^p \leq 2^{2-p}\lambda.$$
\end{lemma}

In the proof of Inequality (i), we will approximate $\xi$ by a sequence 
of simple functions, i.e.
$$\int_{S} \xi(s,z) \tilde \eta(dz,ds)=\lim_{n\to\infty} \int_{S} \xi_n(s,z) \tilde \eta(dz,ds),
$$
where $\xi_n\to\xi$ in $\CM^p(\RR_+; L^p(\nu, E))$.

Therefore, we first show that the inequality is valid for a simple function, and
then extend the inequality to the completion of the set of simple functions, i.e. to all progressively measurable
functions which are $L^p$-integrable. 
Thus,
we suppose here and hereafter that $\xi$ is a simple function. In particular,
we suppose that $\xi$ has the following representation
$$
\xi=\sum_{k=1}^{K}1_{(s_{k-1}, s_k]}(s)\sum_{j}^ J \sum_i ^ I \xi ^k_{ji}\, 1_{A^k_{ji}\times B^k_j}
$$  
with $s_k-s_{k-1}=\tau>0$, $\xi^k_{ji}\in E$, 
$A^k_{ji}\in \mathcal{F}_{s_{k-1}}$, $k=1,\ldots, K$, and $B_j\in\CB(S_\ep)$,
 the finite families of sets $({A_{ji}\times B_j})$ and $({ B_j})$ being pair-wise disjoint and $\nu(B_j)\le 1$.   
 Let us notice that
\DEQS
\int_0^T  \int_S \xi(s, x)\tilde{\eta}(dx,ds)= 
\sum_k  ^ K\sum_{j}^ J  \left(\sum_{i} ^ I 1_{A^k_{ji} } \, \xi^k_{ji}\right)\, \tilde  \eta( B_j\times (s_{k-1},s_k] ).
\EEQS
The Burkholder-Davis-Gundy   inequality, see Inequality \ref{burkholder-simple}, gives
\DEQS
\lqq{\EE \lk| \int_0^T  \int_S \xi(s, x)\tilde{\eta}(dx,ds)\rk|^p \le } &&\\&&
C_p(E)\, \EE\lk( \sum_k ^ K \sum_{j}  ^ J \EE \lk|\sum_{i} ^ I  1_{A^k_{ji} } \, \xi^k_{ji}\,  \tilde  
                \eta( B_j\times (s_{k-1},s_k] )\rk|^p\rk).
\EEQS
Recall that for fixed $k$, the family $\{A_{ji}^k, i=1,\ldots\}$ consists of  disjoint sets. This implicates that 
for fixed $\omega\in\Omega$ only one term of the 
sum over the index $i$ will be not equal to zero. Therefore, we can write
\DEQSZ\label{tocalcu-r}\nonumber 
\lqq{\EE \lk| \int_0^T  \int_S \xi(s, x)\tilde{\eta}(dx,ds)\rk|^p \le } &&\\&&
C_p(E)\,  \sum_k  ^ K\sum_{j} ^ J  \sum_{i} ^ I \EE \lk|  1_{A^k_{ji} } \, \xi^k_{ji} \, \tilde  
                \eta( B_j\times (s_{k-1},s_k] )\rk|^p.
\EEQSZ
In the next step we the fact that $ \eta( B_j\times (s_{k-1},s_k] )$ is Poisson  distributed with 
parameter $\nu(B_j)\tau$.
Therefore,  \eqref{tocalcu-r} reads 
\DEQS
C_p(E)\,  \EE \sum_k ^K \sum_{j} ^J \sum_{i}^I   1_{A^k_{ji} } \lk|  \;\xi^k_{ji}   \EE \lk[  \lk( \sum_{l=1}^\infty 
 1_{\{ \eta( B_j\times (s_{k-1},s_k] ) =l\}} \; l
        -  \lambda_j \rk)  \rk|^p   \mid \CF_{k\tau} \rk]
\EEQS
(note $\CF_{k\tau}=\CF_t$ for $t=k\tau$) and Lemma \ref{lem-cw6} gives
\DEQS
\EE \lk| \int_0^T  \int_S \xi(s, x)\tilde{\eta}(dx,ds)\rk|^p \le C_p(E)\;2^{2-p} \, \EE  \sum_k ^K \sum_{j} ^J \sum_{i}^I
1_{A^k_{ji} } \lk|   \;\xi^k_{ji} \rk|^p \nu(B_j)\tau  .
\EEQS
Going back we arrive at 
\DEQSZ\label{mmmm} 
\\\nonumber
\EE \lk| \int_0^T  \int_S \xi(s, x)\tilde{\eta}(dx,ds)\rk|^p \le C_p(E)\; 2^{2-p} \, \EE\int_0^T \int_S \lk| \xi(s,z)\rk|^p\nu(dz)\, ds.
\EEQSZ
Now, assume that $\xi$ is a progressively measurable process such that
$$
\EE
\int_0^\infty \int_S |\xi(s,z)|^p\nu(dz)\, ds <\infty.
$$
Due to Lemma 1.1, Chapter 1 in \cite{DaP}, there exists a sequence of simple functions $(\xi_n)_{n\in\NN}$
such that $\xi_n\to \xi$ in $\CM^p(\RR_+; L^p(\nu, E))$. Now, due to the definition of the stochastic integral we have 
\DEQS
\EE \lk| \int_0^T  \int_S \xi(s, x)\tilde{\eta}(dx,ds)\rk|^p =
\EE \lk| \lim_{n\to\infty}  \int_0^T  \int_S \xi_n(s, x)\tilde{\eta}(dx,ds)\rk|^p .
\EEQS
The continuity of the norm and inequality \eqref{mmmm} imply
\DEQS
\lqq{\hspace{-2.5cm} \EE \lk| \lim_{n\to\infty}  \int_0^T  \int_S \xi_n(s, x)\tilde{\eta}(dx,ds)\rk|^p =
 \lim_{n\to\infty}  \EE \lk| \int_0^T  \int_S \xi_n(s, x)\tilde{\eta}(dx,ds)\rk|^p 
 }
&& \\
 \phantom{mmmmm}&=& \lim_{n\to\infty}   C_p(E)\; 2^{2-p} \, \EE\int_0^T \int_S \lk| \xi_n(s,z)\rk|^p\nu(dz)\, ds
\\ & = & C_p(E)\; 2^{2-p} \, \EE\int_0^T \int_S \lk| \xi(s,z)\rk|^p\nu(dz)\, ds .
\EEQS
\end{proof}

\begin{proof}[Proof of Inequality (ii)] 
The integrals in inequality  (ii) will be approximated first by 
the omitting the small jumps, i.e.\ by the following limits 
$$
\int_0^t\int_S \xi(s,z) \tilde \eta(dz,ds) = \lim _{\ep\to 0}\int_{S^\ep} \xi(s,z) \tilde \eta(dz,ds)
$$
and
$$
\int_0^t\int_S |\xi(s,z)|^p  \eta(dz,ds) = \lim _{\ep\to 0}\int_{S^\ep} |\xi(s,z)|^p \eta(dz,ds)
$$
where $S_\ep :=S \setminus  B_S(\ep)\footnote{$B_S(y):= \{ x\in S, |x|\le y\}$}$.
Secondly the integrand will be approximated by a sequence of simple functions, i.e.
$$\int_{S^\ep} \xi(s,z) \tilde \eta(dz,ds)=\lim_{n\to\infty} \int_{S^\ep} \xi_n(s,z) \tilde \eta(dz,ds),
$$
where $\xi_n\to\xi$ in $\CM^r(\RR_+; L^r(\nu, E))$.

Before starting with the proof, let here and hereafter $\ep>0$ be fixed. 
Also, we suppose here and hereafter that $\xi$ is a simple function. In particular,
we suppose that $\xi$ has the following representation
$$
\xi=\sum_{k=1}^{K}1_{(s_{k-1}, s_k]}(s)\sum_{j}^ J \sum_i ^ I \xi ^k_{ji}1_{A^k_{ji}\times B^k_j}
$$  
with $s_k-s_{k-1}=\tau>0$, $\xi^k_{ji}\in E$, 
$A^k_{ji}\in \mathcal{F}_{s_{k-1}}$, $k=1,\ldots, K$, and $B_j\in\CB(S_\ep)$,
 the finite families of sets $({A_{ji}\times B_j})$ and $({ B_j})$ being pair-wise disjoint and $\nu(B_j)\le 1$.   
 Let us notice that
\DEQS
\int_0^T  \int_{S_\ep}  \xi(s, x)\tilde{\eta}(dx,ds)= 
\sum_k  ^ K\sum_{j}^ J  \left(\sum_{i} ^ I 1_{A^k_{ji} }  \xi^k_{ji}\right) \tilde  \eta( B_j\times (s_{k-1},s_k] ).
\EEQS
Let $r\ge p$, with $r=p^ m$, $m\in\NN$, be fixed. The Burkholder-Davis-Gundy   inequality, i.e.\ inequality \ref{burkholder-ineq}
with $\Phi(x)=x ^ r$, $x\ge 0$, gives
\DEQS
\lqq{\EE \lk| \int_0^T  \int_{S_\ep} \xi(s, x)\tilde{\eta}(dx,ds)\rk|^r \le } &&\\&&
C_p(E)\, \EE\lk( \sum_k ^ K \sum_{j}  ^ J \lk|\sum_{i} ^ I  1_{A^k_{ji} }  \xi^k_{ji}  \tilde  
                \eta( B_j\times (s_{k-1},s_k] )\rk|^p\rk)^{\frac rp}.
\EEQS
Recall that for fixed $k$ and $j$, $A_{ji}^k$, $i=1,\ldots, I$ are disjoint sets. This implicates that 
only one term of the inner
sum will be not equal to zero. Therefore, we can write
\DEQSZ\label{tocalcu}\nonumber 
\lqq{\EE \lk| \int_0^T  \int_{S_\ep} \xi(s, x)\tilde{\eta}(dx,ds)\rk|^r \le } &&\\&&
C_p(E)\,  \EE\lk( \sum_k  ^ K\sum_{j} ^ J  \sum_{i} ^ I \lk|  1_{A^k_{ji} }  \xi^k_{ji}  \tilde  
                \eta( B_j\times (s_{k-1},s_k] )\rk|^p\rk)^{\frac rp}.
\EEQSZ
Plugging in the definition of $\tilde  
                \eta( B_j\times (s_{k-1},s_k] )$, 
the RHS of \eqref{tocalcu} reads 
\DEQS
C_p(E)\,  \EE\lk( \sum_k ^K \sum_{j} ^J \sum_{i}^I\lk|  1_{A^k_{ji} } \;\xi^k_{ji}  \lk( \sum_{l=1}^\infty  1_{\{
\eta( B_j\times (s_{k-1},s_k] ) =l\}} \; l
        -   \nu(B_j)\tau \rk)  \rk|^p\rk)^{\frac rp}.
\EEQS
Using $|x-y|^p\le 2 ^ p(|x|^p+|y|^p)$  we obtain
\DEQS
& & \ldots \, \le 
C_p(E) \, 2^r\, 
\\ &&\hspace{-0.5cm} 
\EE\lk( \sum_k ^ K \sum_{j} ^ J  \sum_{i} ^ I\lk| \sum_{l=1}^\infty  1_{A^k_{ji} }  1_{\{\eta( B_j\times (s_{k-1},s_k] ) =l\}} 
 l \xi^k_{ji}    
  \rk|^p+  \sum_k ^ K \sum_{j} ^ J  \sum_{i} ^ I \lk| 1_{A^k_{ji} }\xi^k_{ji}    \nu(B_j)\tau  \rk|^p\rk)^{\frac rp}
\\&\le & C_p(E) \, 2 ^ r
\\ &&\hspace{-0.5cm}  \EE\lk( \sum_k ^ K \sum_{j} ^ J  \sum_{i}  ^ I\sum_{l=1}^\infty  1_{A^k_{ji} } 
1_{\{ \eta( B_j\times (s_{k-1},s_k] ) =l\}}  l ^p\lk| \xi^k_{ji}    
  \rk|^p+  \sum_k ^ K \sum_{j}  ^ J \sum_{i} ^ I 1_{A^k_{ji} } \lk| \xi^k_{ji}    \nu(B_j)\tau  \rk|^p\rk)^{\frac rp}.
\EEQS 
Using $|x-y|^q\le  2 ^ q \, (|x|^q+|y|^q)$ for $q=\frac rp$ we get
\DEQSZ\label{lhs}\nonumber 
\ldots  &\le &
C_p(E)\, 2 ^ {r+\frac rp} \,   \EE\lk( 
\sum_k ^K\sum_{j}^J  \sum_{i}^I \sum_{l=1}^\infty  1_{A^k_{ji} } 1_
{\{ \eta( B_j\times (s_{k-1},s_k] ) =l\}} \, l ^p\, \lk| \xi^k_{ji}    
  \rk|^p\rk)^{\frac rp}
\\
\nonumber &&\hspace{1.5cm} {}
+C_p(E)\,  2 ^ {r+\frac rp} \,   \tilde \EE\lk(   \sum_k ^K\sum_{j}^J  \sum_{i}^I \lk| 1_{A^k_{ji} }\xi^k_{ji}    \nu(B_j)\tau  \rk|^p\rk)^{\frac rp}.
\EEQSZ 
Let $m_0$ be so large that $p-1\le m_0\frac pr$. 
We  split again the inner term in the inner sum. Doing this we get
\DEQSZ\label{ssss}\nonumber
&\ldots &\le 
 C_p(E)   \,  2 ^ {r+\frac rp} \,  
 \EE\lk(  \sum_k ^ K \sum_{j} ^ J  \sum_{i} ^ I\sum_{l=1}^{m_0-1}
  1_{\{ \eta( B_j\times (s_{k-1},s_k] ) = l\}}\, 1_{A^k_{ji} } \, l^p\, \lk|  \xi^k_{ji}    
  \rk|^p\rk) ^\frac rp 
  \\&&\nonumber \hspace{0.5cm}{} +  C_p(E)\,  2 ^ {r+\frac rp} \,   \EE \lk( \sum_k ^ K \sum_{j} ^ J  \sum_{i} ^ I
\sum_{l=m_0}^\infty  1_{A^k_{ji} } 1_{\{ \eta( B_j\times (s_{k-1},s_k] ) =l\}}  l ^ p \lk|  \xi^k_{ji}    
  \rk|^p\rk) ^\frac rp
 \\ &&{}\hspace{2cm}{}  + C_p(E) \,  2 ^ {r+\frac rp} \,  \tilde \EE\lk( 
   \sum_k ^K\sum_{j}^J  \sum_{i}^I \lk| 1_{A^k_{ji} }\xi^k_{ji}    \nu(B_j)\tau  \rk|^p\rk)^{\frac rp}.
\EEQSZ
\noindent
\noindent
The first term in \eqref{ssss} can be estimated in the following way. 
Since $l\le ( m_0-1)$, we put $(m_0-1)^{(p-1)\frac rp}$ in front of the braket.
Next, we add the additional
terms for $l=m_0,m_0+1,\ldots$ and then change again the representation
\DEQS
&\ldots &\le  \EE\lk( \sum_k ^ K \sum_{j} ^ J  \sum_{i} ^ I  1_{A^k_{ji} } \sum_{l=1}^{m_0-1}
 1_{\{ \eta( B_j\times (s_{k-1},s_k] ) = l\}} \,l^p  \lk|  \xi^k_{ji}    
  \rk|^p \rk) ^\frac rp 
 \\ 
 &=&  (m_0-1)^{(p-1)\frac rp}\, \EE\lk( \sum_k ^ K \sum_{j} ^ J  \sum_{i} ^ I 
\sum_{l=1}^{m_0-1}     1_{\{ \eta( B_j\times (s_{k-1},s_k] ) = l\}} \, 1_{A^k_{ji} }l \;\lk| \xi_n^j\rk|^p 
 \rk)^\frac rp  
 \\ 
 &\le&  (m_0-1)^{(p-1)\frac rp}\, \EE\lk( \sum_k ^ K \sum_{j} ^ J  \sum_{i} ^ I 
\sum_{l=0}^{\infty }    1_{\{ \eta( B_j\times (s_{k-1},s_k] ) = l\}} \, 1_{A^k_{ji} }l \;\lk| \xi_n^j\rk|^p 
 \rk)^\frac rp  
\\ &=&(m_0-1)^{(p-1)\frac rp}\,
  \EE\lk( \sum_k  ^ K\sum_{j} ^ J  \sum_{i} ^ I  1_{A^k_{ji} }  \lk| \xi^k_{ji}  \rk|^p
                \eta( B_j\times (s_{k-1},s_k] )\rk)^{\frac rp}
\\ &=& (m_0-1)^{(p-1)\frac rp}\, \EE\lk( \int_0^T  \int_{S^\ep} |\xi(s,z)|^p \eta(dz,ds)\rk)^\frac rp 
\\ &\le&  (m_0-1)^{(p-1)\frac rp}\, \EE\lk( \int_0 \prod_{k}^K \, \prod_j^J { \nu(B_j)^{l_{kj}}\tau ^{l_{kj}}\over {l_{kj}}!}^T  \int_{S} |\xi(s,z)|^p \eta(dz,ds)\rk)^\frac rp .
 \EEQS
Now we consider the second term of \eqref{ssss}. 
First, the sets $B_j\times (s_{k-1},s_k]$ are disjoint, therefore
the random variables $ \eta( B_j\times (s_{k-1},s_k] )$ independent. Secondly,
the $\NN_0$ valued random variables
$\{ \eta( B_j\times (s_{k-1},s_k] )\}$ are Poisson  distributed with parameter $\nu(B_j)\tau$,
therefore, the  explicit formula of the expectation gives 
\DEQS
\lqq{ 
\EE \lk(  \sum_k ^ K \sum_{j} ^ J  \sum_{i} ^ I
\sum_{l=m_0}^\infty  1_{A^k_{ji} }  1_{\{ \eta( B_j\times (s_{k-1},s_k] ) = l\}} \, l ^ p \lk|  \xi^k_{ji}    
  \rk|^p\rk) ^\frac rp
=} &&\\
&& 
\sum_{\cl\in \tilde \Omega } \PP\lk(  \eta( B_j\times (s_{k-1},s_k] ) = l_{k,j}, 0\le j\le J,0\le k\le K 
\rk)
\\ && \EE \Bigg[ \Big(\sum_k ^ K \sum_{j} ^ J  \sum_{i} ^ I\sum_{l=m_0}^{\infty }
 1_{A^k_{ji} }  1_{\{ l 
 = \cl_{k,j} \}} \, \cl_{k,j} ^ p \lk|  \xi^k_{ji}    
  \Big|^p\rk) ^\frac rp  
  \\ &&\quad\quad \mid \eta( B_j\times (s_{k-1},s_k] ) = \cl_{k,j}, 0\le j\le J,0\le k\le K \Bigg]
\\ &=& 
\sum_{\cl\in \tilde \Omega }    \prod_{k}^K \, \prod_j^J \, { \nu(B_j)^{l_{kj}}\tau ^{l_{kj}}\over {l_{kj}}! }
\, \EE \Bigg[ \Big(\sum_k ^ K \sum_{j} ^ J  \sum_{i} ^ I   \sum_{l=m_0}^\infty 
 1_{A^k_{ji} }  1_{\{ l 
 = \cl_{k,j} \}} \, \cl_{k,j} ^ p \lk|  \xi^k_{ji}    
  \Big|^p\rk) ^\frac rp  
  \\ &&\quad\quad \mid \eta( B_j\times (s_{k-1},s_k] ) = \cl_{k,j}, 0\le j\le J,0\le k\le K \Bigg]
  \EEQS
where we put $\tilde \Omega= \otimes_{k=1}^K \otimes _{j=1}^J \NN_0$.  
Note that, since the sum over $l$ starts at $m_0$,  if $l_{kj}<m_0$, 
$  1_{ l= 
 = l_{k,j} \}} =0$. Therefore, for any $\cl$ which contributes to the sum we can put at least $m_0$  factors of the product in the front. 
 Therefore, let 
$\ce(\cl)_{jk}:= l_{kj}-m_0$ if $l_{jk}\ge m_0$ and  $\ce(\cl)_{jk}:= l_{kj}$ otherwise,
and $\# \cl :=\{\cl_{k,j}\ge m_0: 0\le k\le K,0\le j\le J\}$. 
Now, putting the  factors of the product  in front of the summands gives 
 \DEQS
 \\ &\le &  {\max_{1\le j\le J}  \nu(B_j)^{m_0 } \tau^{m_0} } 
\sum_{\cl\in \tilde \Omega }    \prod_{k}^K \, \prod_j^J 
 { \nu(B_j)^{\ce(\cl)_{kj}}\tau ^{ \ce(\cl)
 _{ki}}\over \cl_{ki}!}
\\ && \EE \Bigg[ \Big(\sum_k ^ K \sum_{j} ^ J  \sum_{i} ^ I   \sum_{l=m_0}^\infty 
 1_{A^k_{ji} }  1_{\{ l 
 = \cl_{k,j} \}} \, \cl_{k,j} ^ p\lk|  \xi^k_{ji}    
  \Big|^p\rk) ^\frac rp  
  \\ &&\quad\quad \mid \eta( B_j\times (s_{k-1},s_k] ) = \cl_{k,j}, 0\le j\le J,0\le k\le K \Bigg]
  \EEQS
  \DEQS  &=&
  {\max_{1\le j\le J}  \nu(B_j)^{m_0} \tau^{m_0} } 
\sum_{\cl\in \tilde \Omega }    \prod_{k}^K \, \prod_j^J
 { \nu(B_j)^{\ce(\cl)_{kj}}\tau ^{\ce(\cl)_{ki}}\over \ce(\cl)_{ki}!}  
\\ && \EE \Bigg[ \Big(\sum_k ^ K \sum_{j} ^ J  \sum_{i} ^ I   \sum_{l=m_0}^\infty 
 1_{A^k_{ji} }  1_{\{ l 
 = \cl_{k,j} \}} \,  {\cl_{k,j} ^ p \over \cl_{k,j}^\frac pr  (\cl_{kj}-1)^\frac pr \cdots (\cl_{kj}-m_0)^\frac pr }  \lk|  \xi^k_{ji}    
  \Big|^p\rk) ^\frac rp  
  \\ &&\quad\quad \mid \eta( B_j\times (s_{k-1},s_k] ) = \cl_{k,j}, 0\le j\le J,0\le k\le K \Bigg]
.
  \EEQS
 If $p-m_0 \frac pr \le 1$ there exists a constant $C>0$ such that
 $${l ^ {p}\over l^\frac pr  (l-1)^\frac pr \cdots (l-m_0)^\frac pr } \le C\,(l-m_0),\quad l\ge m_0.
 $$ 
Hence, 
   \DEQS \cdots  &\le &C\,
  {\max_{1\le j\le J}  \nu(B_j)^{m_0} \tau^{m_0} } 
\sum_{\cl\in \tilde \Omega }    \prod_{k}^K \, \prod_j^J
 { \nu(B_j)^{\ce(\cl)_{kj}}\tau ^{\ce(\cl)_{ki}}\over \ce(\cl)_{ki}!}  
\\ && \EE \Bigg[ \Big(\sum_k ^ K \sum_{j} ^ J  \sum_{i} ^ I   \sum_{l=m_0}^\infty 
 1_{A^k_{ji} }  1_{\{ l 
 = l_{k,j} \}} \,  {(l-m_0) }  \lk|  \xi^k_{ji}    
  \Big|^p\rk) ^\frac rp  
  \\ &&\quad\quad \mid \eta( B_j\times (s_{k-1},s_k] ) = \cl_{k,j}, 0\le j\le J,0\le k\le K \Bigg]
.
  \EEQS
    Renumerating gives
   \DEQS \cdots  &\le &C\,
  {\max_{1\le j\le J}  \nu(B_j)^{m_0} \tau^{m_0} } 
\sum_{\cl\in \tilde \Omega }    \prod_{k}^K \, \prod_j^J
 { \nu(B_j)^{\cl_{kj}}\tau ^{\cl_{ki}}\over \cl_{ki}!}  
\\ && \EE \Bigg[ \Big(\sum_k ^ K \sum_{j} ^ J  \sum_{i} ^ I   \sum_{l=0}^\infty 
 1_{A^k_{ji} }  1_{\{ l 
 = \cl_{k,j} \}} \,  {l }  \lk|  \xi^k_{ji}    
  \Big|^p\rk) ^\frac rp  
  \\ &&\quad\quad \mid \eta( B_j\times (s_{k-1},s_k] ) = \cl_{k,j}, 0\le j\le J,0\le k\le K \Bigg]
.
  \EEQS
  Going back gives 
   \DEQS \cdots  &\le &C\,
  {\max_{1\le j\le J}  \nu(B_j)^{m_0} \tau^{m_0} } 
\sum_{\cl\in \tilde \Omega }    \prod_{k}^K \, \prod_j^J
 \PP\lk(  \eta( B_j\times (s_{k-1},s_k] ) = \cl_{k,j}, 0\le j\le J,0\le k\le K 
\rk) 
\\ && \EE \Bigg[ \Big(\sum_k ^ K \sum_{j} ^ J  \sum_{i} ^ I   \sum_{l=0}^\infty 
 1_{A^k_{ji} }  1_{\{ l 
 = \cl_{k,j} \}} \,  {l }  \lk|  \xi^k_{ji}    
  \Big|^p\rk) ^\frac rp  
  \\ &&\quad\quad \mid \eta( B_j\times (s_{k-1},s_k] ) = \cl_{k,j}, 0\le j\le J,0\le k\le K \Bigg]
\EEQS
\DEQS
  &\le &C\,
  {\max_{1\le j\le J}  \nu(B_j)^{m_0} \tau^{m_0} } 
\sum_{\cl\in \tilde \Omega }    \prod_{k}^K \, \prod_j^J
 { \nu(B_j)^{\cl_{kj}}\tau ^{\cl_{ki}}\over \ce(\cl)_{ki}!}  
\\ && \EE \Bigg[ \Big(\sum_k ^ K \sum_{j} ^ J  \sum_{i} ^ I   \sum_{l=0}^\infty 
 1_{A^k_{ji} }  1_{\{ l 
 = l_{k,j} \}} \,  {l  }  \lk|  \xi^k_{ji}    
  \Big|^p\rk) ^\frac rp  
  \\ &&\quad\quad \mid \eta( B_j\times (s_{k-1},s_k] ) = l_{k,j}, 0\le j\le J,0\le k\le K \Bigg]
\EEQS
\DEQS
  &=& C\, \, {\max_{1\le j\le J}  \nu(B_j) ^ {m_0} \,\tau ^ {m_0} }\, 
\EE \lk( \int_0^t \int_{S^\ep} |\xi(z,s)|  \,   \eta( dz, ds  ) \rk) ^\frac rp.
  \EEQS
Using the assumption $ \nu(B_j)\le 1$ for all $j=1,\ldots, J$,
we obtain 
\DEQS
  &=& C\, \, \,\tau ^ {m_0} \, 
\EE \lk( \int_0^t \int_{S} |\xi(z,s)|  \,   \eta( dz, ds  ) \rk) ^\frac rp.
  \EEQS
It remains to investigate the last summand in \eqref{lhs}.
Observe that, first, since $\nu(B_j)\le 1$, $\nu(B_j)^p\le \nu(B_j)$, and, secondly, $\sum_j ^ J\nu(B_j)=\nu(S_\ep)$.
Thus, applying 
the H\"older inequality twice, and then, again, taken into account that $\{ A^k_{ji}, 1\le i\le I\}$ are disjoint,
we put the sum running over $i$ in front of  the brackets.
Doing so, we arrive at
\DEQSZ\label{aaabove}
\\
\nonumber
\lqq{ \vspace{-2cm} \tilde \EE\lk(  
 \sum_k ^K \sum_{j}^J  \sum_{i}^I\lk| 1_{A^k_{ji} }\, \xi^k_{ji}    \nu(B_j)\tau  \rk|^p\rk)^{\frac rp}
\le  
\tilde \EE\lk(   \sum_k ^ K  \sum_{j} ^ J \sum_{i}^I \tau ^p 1_{A^k_{ji} } \lk| \xi^k_{ji}    \rk|^p \nu(B_j) ^ p  \rk)^{\frac rp}
\phantom{nnnnnnn} } &&
\\ &\le &\nonumber
 \tau^ {(p-1)\frac rp } \, \max_j\nu(B_j) ^   {(p-1)\frac rp } \,
\,  \nu(S_\ep) ^ \frac pr \,
\tilde \EE   \sum_k ^ K \tau  \lk( \sum_{j} ^ J  \sum_{i}^I  1_{A^k_{ji} }\lk| \xi^k_{ji}    \rk|^p {\nu(B_j)\over \nu(S_\ep) }  \rk)^{\frac rp}
\\ &\le &\nonumber
 \tau^{(p-1)\frac rp } \, \max_j\nu(B_j) ^   {(p-1)\frac rp }\,  \nu(S_\ep) ^ \frac pr \,
\tilde \EE   \sum_k ^ K \tau \sum_{j} ^ J  \lk(  \sum_{i}^I  1_{A^k_{ji} }\lk| \xi^k_{ji}    \rk|^p \rk)^{\frac rp}{\nu(B_j)\over \nu(S_\ep) }  
\\ &\le& \nonumber
 \tau^{r-\frac rp }  \, \max_j\nu(B_j) ^   {(p-1)\frac rp }\,  \nu(S_\ep)  ^ \frac pr \tilde \EE \sum_k  ^K\tau  \sum_{j} ^ J  \sum_{i} ^ I  1_{A^k_{ji} }\lk| \xi^k_{ji}    \rk|^r   {\nu(B_j)\over \nu(S^\ep) } .
\\ &\le& \nonumber
 \tau^{r-\frac rp }  \, \max_j\nu(B_j) ^   {(p-1)\frac rp }\,  \nu(S_\ep)  ^ {\frac pr-1}
 \tilde \EE \sum_k  ^K\tau  \sum_{j} ^ J  \sum_{i} ^ I  1_{A^k_{ji} }\lk| \xi^k_{ji}    \rk|^r   {\nu(B_j) } .
\EEQSZ
The RHS of \eqref{aaabove}
is bounded by
$$
\nu(S_\ep) ^{ \frac pr -1} \, \tau^{r-\frac rp}\, \int_0 ^ T \int_{S_\ep} \EE\lk| \xi(s,z)\rk| ^ r\nu(dz)\, ds.
$$
Collecting all together we arrive  at
\DEQS 
\lqq{\hspace{2cm}\EE \lk| \int_0^T  \int_{S^\ep}  \xi(s, x)\tilde{\eta}(dx,ds)\rk|^r \le 
}&&\\&& C_p(E) 2 ^ {r+\frac rp}    \,(1+\tau)\nonumber
 \EE\lk( \int_0^T \int_{S}  |\xi(s,x)|^p\, \eta(dx,ds)\rk)^\frac rp 
 \\&&{} +
 C_p(E) 2 ^ {r+\frac rp}  \nu(S^\ep) ^{ \frac pr -1} \tau^{p-1} \int_0 ^ T \int_ {S}  \EE\lk| \xi(s,z)\rk| ^ r\nu(dz)\, ds
.
\EEQS 
It remains to take the limit. But, since $r>p$, $\nu(S_\ep) ^{ \frac pr-1}\to 0$ as $\ep\to 0$, we obtain  
\DEQSZ\label{fffinal}
\lqq{\hspace{1cm}\EE \lk| \int_0^T  \int_{S}  \xi(s, x)\tilde{\eta}(dx,ds)\rk|^r \le 
}&&\\&& C_p(E) \, 2 ^ {r+\frac rp}    \,(1+\tau)\, \nonumber
 \EE\lk( \int_0^T \int_{S}  |\xi(s,x)|^p\, \eta(dx,ds)\rk)^\frac rp 
\EEQSZ
\noindent
In the second step we assume that $\xi\in \CM^r(\RR_+;L^r(S,\nu;E))$ is approximated by a sequence of 
simple functions $(\xi_n)_{n\in\NN}$,
where we take in time the shifted Haar projection of order $n$ and in space an arbitrary simple function.
Therefore, let $\xi_n$, $n\in\NN$, be a sequence of simple functions, such that $\xi_n$ is constants on the 
dyadic intervals $[2^{-n}k, 2^{-n}(k+1))$ and $\xi_n\to\xi$ in $\CM^r(\RR_+;L^r(S,\nu;E))$.
Substituting $\xi_n$ in \eqref{fffinal} we obtain 
\DEQS
\lqq{\EE \lk| \int_0^T  \int_{S_\ep} \xi(s, x)\tilde{\eta}(dx,ds)\rk|^r \le 
 C_p(E)\, 2 ^ {r+\frac rp}  \,(1+ 2^{-n})\times } &&\\&&
\\ &&\lk\{   \EE\lk( \int_0^T \int_{S} |\xi(s,x)|^p\, \eta(dx,ds)\rk)^\frac rp 
+  \nu(S_\ep) ^{ \frac pr -1} 2^{-n(p-1)} \int_0 ^ T \int_S \EE\lk| \xi(s,z)\rk| ^ r\nu(dz)\, ds  \rk\}
. 
\EEQS
Taking the limit for $n$ to infinity we get
\DEQSZ\label{lllast}
\\\nonumber\EE \lk| \int_0^T  \int_{S_\ep}  \xi(s, x)\tilde{\eta}(dx,ds)\rk|^r \le 
 C_p(E)\, 2 ^ {r+\frac rp}\, \EE\lk( \int_0^T \int_{S} |\xi(s,z)|^p\, \eta(dz,ds)\rk)^\frac rp
.
\EEQSZ
In the third and last step we let $\ep$ converge 
 to  zero. Since the RHS of \eqref{lllast} is independent of $\ep$, the  assertion is shown.
\end{proof}

\begin{proof}
[Proof of Inequality (iii):] 
The proof is a generalisation of the proofs of Bass and Cranston
\cite[Lemma 5.2]{MR88b:60113}
or Protter and Talay \cite[Lemma 4.1]{MR98c:60063}.
It follows from Inequality (ii) 
that 
\DEQS
\lqq{ \EE\sup_{0<s\le t} \lk|   \int_{0 } ^t\int_S \xi(s;z)\tilde\eta(dz;ds)
\rk| ^{p ^n}} &&\\
 &\le&  C_p(E) \, 2 ^ {p ^ n+p ^ {n-1}}  (m_0-1)^ {(p-1)\frac rp} \:
 \EE \lk(    \int_{0 } ^t
\int_S |\xi(s;z)| ^p \; \eta(dz;ds )\rk)^{p ^{n-1}}.
\EEQS
Simple calculations lead to
\DEQSZ\label{hhhhmuss}\nonumber 
\lefteqn{ \EE\sup_{0<s\le t} \lk|   \int_{0 } ^t\int_S \xi(s;z)\tilde \eta(dz;ds)
\rk| ^{p ^n}\le   C_p(E) \,  2 ^ {p ^ n+2 p ^ {n-1}} \,  (m_0-1)^ {(p-1)\frac rp}\, \times 
}
&&
\\\nonumber
&\le &
\lk( \EE \lk(
       \int_{0 } ^t\int_S |\xi(s,z)|^p  \;\tilde \eta (dz;ds)  \rk)^{p ^{n-1}}
 + \EE \lk( \int_{0 } ^t\int_S  |\xi(s,z)|^p  \;\gamma (dz;ds)\rk)^{p ^{n-1}}
\rk)
\EEQSZ
Let us define
$$
L(t) ^{(0)} := \int_{0 } ^t \int_S |\xi(s,z)|^p  \; \tilde\eta 
(dz;ds),\quad t\ge 0.
$$
Then,
\DEQSZ\label{hhhhmus}\nonumber 
\lefteqn{ \EE\sup_{0<s\le t} \lk|   \int_{0 } ^t\int_S \xi(s;z)\tilde \eta(dz;ds)
\rk| ^{p ^n}\le   C_p(E) \,  2 ^ {p ^ n+2 p ^ {n-1}} \,  (m_0-1)^ {(p-1)\frac rp}\,\times   
}
&&\\
&&
\; \lk( \EE  |L(t) ^{(0)}| ^{p ^{n-1}}
 + \EE \lk( \int_{0 } ^t\int_S   |\xi(s,z)|^p  \: \nu(dz)\: ds\rk)^{p ^{n-1}}
\rk).
\EEQSZ
If $n$ equals $2$ we are done. In particularly,
Inequality (i) for $r=p$ gives
%
\DEQSZ\label{result}\nonumber
\EE\lk| L(t) ^{(0)}\rk| ^p  &\le &   \EE\lk|
\int_{0 } ^t \int_S |\xi(s,z)|^p  \; \tilde \eta (dz;ds)\rk| ^p
\\ &\le &  2^{2-p} \, \EE
\int_{0 } ^t \int_S  |\xi(s,z)| ^{p^2}   \; \nu 
(dz) \, ds.
\EEQSZ
Substituting \eqref{result} in \eqref{hhhhmus}  we get for $n=2$
\DEQS
\lefteqn{ \EE\sup_{0<s\le t} \lk|   \int_{0 } ^t\int_S \xi(s;z)\tilde \eta(dz;ds)
\rk| ^{p ^2}\le C_p(E)\, 2 ^ {p ^ 2+2p}\, \times 
}
&&
\\ &&
 \lk(  2^{2-p} \,\EE\int_{0 } ^t  \int_S |\xi(s,z)| ^{p ^2}   \; \nu (dz)\: ds 
 + \EE \lk( \int_{0 } ^t\int_S  |\xi(s,z)|^p  \: \nu(dz)\: ds\rk)^{p }
\rk).
\EEQS
Now, Inequality (iii) is proved, provided $n=2$.
If $n>2$, then we have to continue. Let 
$$
L(t) ^{(r)} := \int_{0 } ^t \int_\RR |\xi(s,z)| ^{p ^{r+1}} \;\tilde \eta 
(dz;ds)
\quad \mbox{ for } r=1,\ldots, n.
$$
Inequality (ii) 
leads to 
\DEQS
\lefteqn{
\EE\lk| L(t) ^{(r)}\rk| ^{p  ^{m}} =   
 \EE \lk( \int_0 ^t \int_\RR
|\xi(s,z)|^{p ^{r+1}}\tilde  \eta\: (dz;ds) \rk) ^{ p  ^{m}}
} &&
\\
&\le &
C_p(E)\, 2 ^ {p ^ m+p ^ {m-1}} \,  (m_0-1)^ {(p-1)\frac rp}  \; \EE \lk( \int_0 ^t \int_\RR
|\xi(s,z)| ^{p ^{r+2}}
\eta  (dz;ds)\rk) ^{ p  ^{m-1}}.
\EEQS
Since $\gamma=\nu\times \lambda$, simple calculations lead to
\DEQSZ\label{inlpm}\nonumber
\lefteqn{
\EE\lk| L(t) ^{(r)}\rk| ^{p  ^{m}} \le  C_p(E)\, 2 ^ {p ^ m+p ^ {m-1}}  } 
&&
\\\nonumber
& &
  \; \EE \lk( \int_0 ^t \int_\RR
|\xi(s,z)|  ^{p ^{r+2}}
\tilde \eta  (dz;ds)+
\int_0 ^t \int_S
|\xi(s,z)|  ^{p ^{r+2}}
\gamma  (dz;ds)\rk) ^{ p  ^{m-1} }
\\\nonumber 
&\le &
  C_p(E)\, 2 ^ {p ^ m+2 p ^ {m-1} }   (m_0-1)^ {(p-1)\frac rp}  \phantom{\Big|}
\\ &&\hspace{1cm}  \;\lk(  \EE  \lk|L(t) ^{(r+1)}\rk| ^{p ^{m-1}}
 +  \EE\lk( \int_{0 } ^t\int_S  |\xi(s;z)| ^{p ^{r+2}} \;\nu (dz)\, ds\rk)^{p ^{m-1}}\rk)
.
\EEQSZ
Starting with $ L(t) ^{(0)}$ and 
iterating the calculation done in (\ref{inlpm}) lead for arbitrary $n$ to
\DEQSZ\label{theatabove}\nonumber
\lqq{ \EE\lk| L(t) ^{(0)}\rk| ^{p  ^{n}}
\le   C_p(E)\, 2 ^ {p ^ n+2 p ^ {n-1} }   (m_0-1)^ {(p-1)\frac rp} }
&&
\\ && \hspace{1cm}  \;\lk(  \EE  \lk|L(t) ^{(1)}\rk| ^{p ^{n-1}}\nonumber
 +  \EE\lk( \int_{0 } ^t\int_S  |\xi(s;z)| ^{p } \;\nu (dz)\,ds\rk)^{p ^{n-1}}\rk)
\\ &\le &\nonumber
  C_p(E)\, 2 ^ {p ^ n+2 p ^ {n-1} }    \, 2 ^ {p ^ {n-1}+2 p ^ {n-2} }  (m_0-1)^ {(p-1)\frac rp} \;  \EE  \lk|L(t) ^{(2)}\rk| ^{p ^{n-2}}
\\&&\nonumber
{} + 
\sum_{i=1} ^2 \bar C(i) \,  \EE\lk( \int_{0 } ^t\int_S  |\xi (s;z)| ^{p ^{i}} \;\nu (dz)ds\rk)^{p ^{n-i}} 
\\
&\le &  
C_p(E)  2 ^ {p ^ n+2 p ^ {n-1} }   (m_0-1)^ {(p-1)\frac rp}  \; \EE  \lk|L(t) ^{(n-1)}\rk| ^{p }+
\\ && \sum_{i=1} ^{n-1} \bar C(i) 
\,     \EE \lk( \int_{0  } ^t \int_S  | \xi(s,z)| ^{p ^{i}}\; \nu(z) ds\rk) ^{p^{n-i}},
\EEQSZ
where $\bar C(0)=1$ and $\bar C(i)= \bar C(i-1)2 ^ { p ^ {n-i+1}+2p ^ {n-1}} (m_0-1)^ {(p-1)\frac rp}\, m(n-i)  $.
Finally, 
Inequality (i) gives 
\DEQSZ \label{sub01}
\EE\lk| L(t) ^{(n-1)}\rk| ^{p  } &= &
 \EE \lk| \int_{0 } ^t\int_S |\xi(s,z)|  ^{p ^{n}} \; \tilde\eta (dz;ds)\rk| ^p
\nonumber
\\
&\le &
2 ^{2-p} \;\EE \int_{0} ^t\int_S |\xi(s,z)|  ^{p ^{n+1}} \nu(dz)\, ds.
\EEQSZ
Thus, substituting \eqref{sub01} in \eqref{theatabove} we arrive at 
\DEQS
\lefteqn{ \EE\sup_{0<s\le t} \lk|   \int_{0 } ^t\int_S \xi(s;z)\tilde \eta(dz;ds)
\rk| ^{p ^n}\le \;C_p(E)  2 ^ {p ^ n+2 p ^ {n-1} }   (m_0-1)^ {(p-1)\frac rp} 
}
&&\\
& &
\; \lk( \EE  |L(t) ^{(0)}| ^{p ^{n-1}}
 + \EE \lk( \int_{0 } ^t\int_S   |\xi(s,z)|^p  \: \nu(dz)\: ds\rk)^{p ^{n-1}}
\rk)
\\&
\le&
  2^{2-p} \;\sum_{i=1} ^n\bar C(i)\,
     \EE \lk( \int_{0  } ^t \int_S  | \xi(s,z)| ^{p ^{i}}\; \nu(z) ds\rk) ^{n-i}.
\EEQS
\end{proof}

\appendix

\section{Discrete Inequalities of Burkholder-Davis-Gundy type}

\del{Sources \cite{burkholdergundy}

Before starting we will introduce the following notation. 
For any $E$-valued martingale over a probability space $(\Omega,\CF,\PP)$, let
$\{\CF^M_n\}_{n=1}^N$ the natural filtration induced by $\{M_{n}\}_{n=0}^N$, i.e.\
$\CF^M_n :=\sigma( M_k, k\le n)$, $n\in\NN$. By $\{m_{n}\}_{n=0}^N$ we denote the sequence of martingale differences, i.e.\
$m_n:= M_{n}-M_{n-1}$, $0\le n\le N $ (remember $M_{-1}:=0$), and by $M^\ast $ the sequence 
$\{M^\ast_n\}_{n=0}^N$ given by $M_n^\ast := \sup_{k\le n} |M_n|$, $0\le n\le N$. Let 
$$S_{n,p}(M) := \lk( \sum_{k\le n} |m_k|_E^p\rk)^\frac 1p,\quad 0\le n\le N
$$
with $S_p(M):= S_{N,p}(M)$, and
$$s_{n,p}(M) := \lk( \sum_{k\le n} \EE [|m_k|_E^p\mid \CF_{k-1} ]\rk)^\frac 1p,\quad 0\le n\le N
$$
again, with $s_p(M):= s_{N,p}(M)$.
}

In this section we collect some basic informations about the     martingale type $p$, $p\in [1,2]$, Banach spaces.
For more details we refer to \cite[Appendix C]{maxreg} or \cite{MR1313905}.
A property which encompasses both, the UMD property and the type $p$ property is the martingale type $p$ property.

\begin{definition}\label{def-mart}
Assume that $p\in [1,2]$ is fixed. A Banach space $E$ is of  martingale type $p$  iff
there exists a constant $L_{p}(E)>0$ such that for all
$E$-valued  finite martingale $\{M_{n}\}_{n=0}^N$  the
following inequality holds
\DEQS \sup_{0\le n\le N } \mathbb{E} | M_{n} |_E ^{p} \le  L_{p}(E)\,
\mathbb{E}  \sum_{n=0}^N  | M_{n}-M_{n-1} |_E ^{p},
\EEQS 
 where as  usually, we put  $M_{-1}=0$.
\end{definition}


A useful tool in the theory of martingales is the Doob's maximal inequality.
The simplest version says that all real valued non-negative submartingales $\{M_{n}\}_{n=0}^N$  satisfy
the  inequality,
$$
\lambda  \PP\lk( \sup_{0\le  k\le n} |M_k|>\lambda\rk) \le \EE\, 1_{\max _{k\le n} M_k\ge \lambda } \, |M_n| ,\quad 1\le n\le N,
$$
and, hence, satisfy 
\DEQSZ\label{doob-simple}
\lambda  ^ p \, \PP\lk( \sup_{0\le  k\le n} |M_k| ^ p >\lambda\rk) \le \EE  |M_n| ^ p ,\quad 1\le n\le N.
\EEQSZ
Now, one gets immediately that all real valued  non-negative submartingales $\{M_{n}\}_{n=0}^N$ satisfy
\DEQSZ\label{burkholder-simple}
\EE |M_n| ^ p \le  \EE \sup_{0\le  k\le n} |M_k| ^ p \le q^ p \, \EE |M_n| ^ p,
\EEQSZ
where $q$ is the conjugate exponent to $p$.
From the last version of Doob's maximal inequality  we can derive 
the following corollary.

\begin{corollary}
Let $p\in [1,2]$ and let $E$ be a Banach space of martingale type $p$.
Then there exist a constant $C=C_p(E)$ such that    for all
$E$-valued  finite martingale $\{M_{n}\}_{n=0}^N$  the
following inequality holds
\DEQSZ
 \mathbb{E}\sup_{0\le n\le N } | M_{n} |_E ^{p}  &\le &  C\,
\mathbb{E}  \sum_{k=0}^N  | M_{n}-M_{n-1} |_E ^{p},
\label{eqn-2.2}
\EEQSZ 
 where as  usually, we put  $M_{-1}=0$.
\end{corollary}

\del{\begin{theorem}(Kahane-Khintchin inequality)
Let $(r_n)$ ne a Rademacher sequence, i.e.\ a sequence of independent random variables with $\PP( r_n=-1)=\PP(r_n=+1)=\frac 12$.
For all $1\le p,q<\infty$ there exists a constant $C_{p,q}$, depending on $p$ and $q$, such that
for all finite sequences $x_1,\ldots x_N\in E$ we have
$$ \lk( \EE \lk| \sum_{n=1}^N r_n x_n \rk|^p\rk)^\frac 1p \le C_{p,q} \lk( \EE \lk| \sum_{n=1}^N r_n x_n \rk|^q\rk)^\frac 1q.
$$ 
\end{theorem}
}

Nevertheless, in the proof of inequality (ii) we used a stronger inequality, namely, 
we supposed that there exists a constant $C$ such that
for all
$E$-valued  finite martingales $\{M_{n}\}_{n=0}^N$  the
following inequality holds
\DEQSZ  
  \mathbb{E} | \sum_n  M _n   | ^{r} &\le & C
 \mathbb{E}  \lk( \sum_{n=0}^N  | M_{n}-M_{n-1}  | ^{p}\rk) ^ \frac rp.
\label{eqn-2.3}\EEQSZ

This stronger inequality we can derive from  a generalisation of Doob's maximal inequality.
But before showing inequality \eqref{eqn-2.3}, since it is interesting on its own, we state the generalization of  Doob's maximal inequality.
To be more precise, in the Doob's maximal inequality we can replace the square by a 
convex, non decreasing and continuous function  $\Phi:[0,\infty)\to \RR$
 with $\Phi(0)=0$. In addition, 
$\Phi$ has to satisfy  {\em the growth condition} (see Appendix)

\begin{proposition}\label{doob}(Garsia \cite[p.173]{garsia})
For all convex,  non decreasing, and continuous function $\Phi:[0,\infty)\to \RR$ with $\Phi(0)=0$ and 
satisfying the growth condition, 
there exists a constant $C$ such that 
for all   non-negative real valued sub-martingales   $\{X_{n}\}_{n=0}^N$ with  $X_0=0$
we have
$$
\EE \Phi\Big( \sup_{1\le n\le N} X_n\Big) \le C\, \EE \Phi\big( X_N \big).
$$  
To be precise, $C=4\, (C_\Psi ^ \ast-1)$ where $\Psi $ is the conjugate convex function to $\Phi$.
\end{proposition}

But before starting with the proof we state a  result of Garsia, i.e.\ \cite[Theorem 2.1]{garsia}.
\begin{theorem}\label{theorem.1.}
If   $\{X_{n}\}_{n=0}^N$ is a  nonnegative real valued sub-martingale  with  $X_0=0$ and if
 $\phi:[0,\infty)\to [0,\infty)$ is a non decreasing
function, then
$$
\EE \int_0 ^ {X ^ \ast_n} t\, d\phi(t) = \EE X_n \phi( X ^ \ast_n),\quad n\in\NN, 
$$
where $X ^ \ast _n = \sup_{k\le n} |X_k|$, $n\in\NN$.
\end{theorem}
Now, we can start with the proof.
\begin{proof}[Proof of Proposition \ref{doob}]
Now, since $\Psi(\phi(t))=\int_0^t s \, d\psi(s)$, we obtain
$$
\EE \Psi( \phi( X^\ast_n)) = \EE \int_0^{X^\ast_n} s \, d\psi(s).
$$
Theorem \ref{theorem.1.} gives
$$
\EE \Psi( \phi( X^\ast_n)) = \EE  X_n \phi( {X^\ast_n}) .
$$
The Young inequality gives
\DEQSZ\label{eq:mmm1}
\EE \Psi( \phi( X^\ast_n)) = {2 }\, \EE \Phi( X_n) +\frac 12\, \EE \Psi( \phi({X^\ast_n})) .
\EEQSZ
Subtracting $\frac 12\, \EE \Psi( \phi({X^\ast_n})$ on both sides of inequality \eqref{eq:mmm1} gives
$$
\frac 12\, \EE \Psi( \phi( X^\ast_n))  \le {2}\, \EE \Phi( X_n).
$$
By property \ref{convex} we get 
the assertion.
\end{proof}

\del{\begin{proposition}(\cite[Proposition 3.1, for martingale p.221]{pisier-1206}) 
Let $E$ be a Banach space of type $p$ and cotype $q$.
Then there exists constants $c,C$ such that for any 
$E$-valued  finite sequence of independent random variables $\{x_{n}\}_{n=0}^N$ and any $0<r<\infty$ the
following inequality holds
\begin{equation} c  \mathbb{E} \lk( \sum_n | x _n  | ^{q}\rk)  ^ \frac rq
\le  \mathbb{E} | \sum_n  x _n   | ^{r} \le  C
 \mathbb{E}  \lk( \sum_{n=0}^N  | x_n  | ^{p}\rk) ^ \frac rp,
\label{eqn-2.1}\end{equation}
 where as  usually, we put  $M_{-1}=0$.
\end{proposition}}

Assume $E$ is of martingale type $p$. 
Now, from the 
Definition \ref{def-mart} and the generalized 
Doob maximal inequality, i.e.\ Proposition \ref{doob},
we can show that the 
Burkholder-Davis-Gundy inequality is also valid on $E$.

\begin{theorem}\label{burkholder-ineq}
Let $\Phi:[0,\infty)\to \RR$ be a non decreasing, convex and continuous function with $\Phi(0)=0$ 
and satisfying the growth condition (for definition we refer to \ref{appendix-convex-function}). 
\del{such that
$\Phi$ satisfies the following growth condition: there exists a constants $C>0$ with
$$\Phi( 2\lambda)\le c \Phi(\lambda),\quad \lambda\in [0,\infty).
$$}
\\
Let $p\in [1,2]$ be fixed and let $E$ be a Banach space $E$ of  martingale type $p$. Then, 
there exists a constant $C_{p}(E,\phi)>0$ such that for all
$E$-valued  finite martingale $\{M_{n}\}_{n=0}^N$  the
following inequality holds
\DEQSZ 
\nonumber
 \mathbb{E}\Phi\lk( \sup_{0\le n\le N } | M_{n} |_E \rk)  \le C_{p}(E,\phi)
\, \mathbb{E}\Phi\lk( \lk(\sum_{n=0}^N | M_{n}-M_{n-1} |_E ^p \rk)^\frac 1p \rk),
\label{eqn-2.2-1}
\EEQSZ 
 where as  usually, we put  $M_{-1}=0$.
\end{theorem} 
We will need the following Lemmata. Since the  Lemmata are valid for real valued 
random variables,
we omit their proofs and give only the reference.

\begin{lemma}\label{lemma-convex}(Burkholder, Davis and Gundy \cite[Theorem 3.2]{bdg}, Garsia \cite[Theorem 0.1]{garsia}) 
Let $\Phi$ be a convex function satisfying the conditions of Theorem \ref{burkholder-ineq}
and $(\Omega,\CF,\{\CF_n\}_{n=0}^N,\PP)$ be a filtered probability space.
Then there exists a constant $C$, only depending on $\Phi$, such that 
for all  sequence $\{ z_n\}_{n=0}^N$ of real-valued, non negative and measurable functions  $(\Omega,\CF,\PP)$
 the
following inequality holds
 $$
\EE \, \Phi\lk( \sum_{n=0} ^ N \EE[z_n\mid \CF_{n-1}]\rk)\le C\, \EE\Phi\lk( \sum_{n=0} ^ N  z_n\rk).
$$
To be more precise, $C=(c_\Phi ^ \ast) ^ {2c_\Phi ^ \ast}$, where $c_\Phi ^ \ast$ is defined in \eqref{c-phi}.
\end{lemma}

\begin{proof}[Proof of Lemma \ref{lemma-convex}]
We are following the proof of Garsia \cite[Theorem 0.1]{garsia}.
Put for $\phi$ given by ... 
\DEQS 
Z_n := \sum_{k=0}^n z_k, &&\quad Z^\CF_n := \sum_{k=0}^n \EE[ z_k\mid \CF_k]
\\ 
Y_0=0, && \quad Y_k:=\EE[ \phi( Z_n^\CF)\mid \CF_k],\, 0\le k\le n.
\EEQS
This given we obtain by the tower property
\DEQS
\EE Z^\CF_n \phi( Z^\CF_n )= \EE  \sum_{k=0} ^n \EE[ z_k\mid \CF_k]\cdot \phi( Z_n^\CF)
=\EE   \sum_{k=0} ^n \EE[  \,\EE[ z_k\mid \CF_k]\cdot \phi( Z_n^\CF)\mid \CF_k]
\\=\EE   \sum_{k=0} ^n \EE[ z_k\mid \CF_k] \cdot \EE[ \phi( Z_n^\CF)\mid \CF_k]=
\EE   \sum_{k=0} ^n  z_k \,  \EE[ \phi( Z_n^\CF)\mid \CF_k]\le \EE   \sum_{k=0} ^n  z_k \,  Y_n^\ast
\le  Z_n Y_n^\ast.
\EEQS
From the Young inequality we get for $a=c_\phi^\ast $ (for definition of $c_\phi^\ast$ see \eqref{c-phi})
\DEQS
\EE Z^\CF_n \phi( Z^\CF_n )\le \EE  \Phi(a^2 Z_n) +\EE \Psi( a^{-2} Y_n^\ast)
\EEQS
From properties of $\Phi$, i.e. \eqref{fig:01} we get
\DEQS
\EE \Phi(  Z^\CF_n)+\EE\Psi(  \phi( Z^\CF_n ))\le \EE  \Phi(a^2 Z_n) + a^{-2}\, \EE \Psi( Y^\ast_n)
\le\EE  \Phi(a^2 Z_n) + a^{-1}\, \EE \Psi( a Y_n).
\EEQS
From   and the definition of $Y$, we get
\DEQS
\EE \Phi(  Z^\CF_n)+\EE\Psi(  \phi( Z^\CF_n ))\le (c_\Phi^\ast)^{2c_\phi^\ast }\EE  \Phi( Z_n) +  \EE \Psi( Y_n)
\le (c_\Phi^\ast)^{2c_\phi^\ast }\EE  \Phi( Z_n) +  \EE \Psi( \phi( Z_n^\CF)).
\EEQS
Subtraction on both sides $\EE \Psi( \phi( Z_n^\CF) )$ leads the assertion.

\end{proof}
\begin{lemma}\cite[Lemma 7.1]{burkholder}\label{ausburkholder}
Suppose that $x$ and $y$ are nonnegative  $\RR$-valued random variables on a probability space $(\Omega,\CF,\PP)$
and $\beta>1$, $\delta>1$, $\ep>1$ are real numbers such that
$$
\PP\lk( y> \beta \lambda,\, x\le \delta \lambda\rk) \le \ep \,\PP \lk( y>\lambda\rk),\quad \lambda>0.
$$
In addition, let $\gamma$ and $\eta$ be real numbers satisfying
$$
\Phi( \beta\lambda)\le \gamma \Phi(\lambda),\quad \Phi(\lambda/\delta)\le \eta \Phi(\lambda),\quad \lambda\ge 0.
$$
If $\gamma\ep<1$ then 
$$
\EE \Phi(y)\le {\gamma\eta\over 1-\gamma\ep} \EE \Phi(x).
$$
\end{lemma}
\begin{proof}[Proof of Lemma \ref{ausburkholder}]
The proof  follows by some direct calculations, therefore, we omit the proof and refer the reader 
e.g. to \cite[Lemma 7.1]{burkholder}.
\end{proof}
By means of the generalised Doob's maximal inequality and the Lemmata before, following, for the Proof of Theorem
\ref{burkholder-ineq} necessary,  Proposition can be verified. 
\begin{proposition}\label{inqqq}
There exists a constant $C<\infty$ such that
for all $E$-valued martingales  $\{M_n\}_{n=0} ^N $ and  all $M$-previsible processes $\{w_m\}_{m=0}^N $  
satisfying 
$|M_n-M_{n-1}|\le w_n$
for all $1\le n\le N $,
 we have
\DEQS
\EE\Phi(M ^ \ast _n) \le C\EE \phi(S_{n,p}(M) )+ C \EE \phi (w_n ^ \ast),\quad 1\le n\le N.
\del{\\
\EE\Phi( S_{n,q}(M) ) \le C\EE \phi(  M_n ^ \ast )+ C \EE \phi (w_n ^ \ast),\quad 1\le n\le N.}
\EEQS
An estimate of the constant is given by $C_{\Phi,p} (E)= \min\{ 2 \delta ^ {-c_\Phi ^ \ast} \beta  ^ {c_\Phi ^ \ast}:\, 
\beta>1,0<\delta<1-\beta$ such that $2L_p(E){\delta ^ p\beta  ^ {c_\Phi ^ \ast} \over (\beta-\delta-1) ^ p}=\frac 12\}$.
\end{proposition}
\begin{proof}[Proof of Proposition  \ref{inqqq}]
The proof follows the proof of \cite[Theorem 15.1]{burkholder}, where 
only the real valued case is considered. Therefore, we had to modify the original proof of Burkholder  
at some points.
Without loss of generality we set $n=N$.
Similarly, we will show that the random variables $M_N ^ \ast$ and $S_{N,p}(M)\vee w_N  ^ \ast$ satisfies the assumption of 
Lemma \ref{ausburkholder}.
I.e. we will show that for $\beta>1$ and $0<\delta<\beta-1$ the following holds
\DEQSZ\label{kkk}
\\
\nonumber
\PP\lk( M ^ \ast_N  > \beta \lambda, \, S_{N,p}(M)\vee w ^ \ast_N  \le \delta\lambda\rk) \le
 2\, L_p(X)\,{  \delta  ^ p\over (\beta -\delta-1) ^ p } 
\,\PP\lk( M _N ^ \ast >\lambda\rk),\quad \lambda>0.
\EEQSZ
To prove \eqref{kkk}, we introduce the following stopping times. Let
$
\mu := \inf\{1\le n\le N : |M_n|>\lambda\}$, $\nu := \inf\{1\le n\le N: |M_n|>\beta\lambda\}$, and  
$\sigma := \inf\{1\le n\le N: |S_{n,p}(M)|>\delta\lambda\mbox{ or } w_{n}>\delta\lambda\}$.
If the infimum will not be attained, we set the stopping time to $\infty$.
Let $H=\{ H_n,\, 1\le n\le N\}$ be defined by
$$ H_n := \sum_{k\le n} h_k=\sum_{k\le n} 1_{\{ \mu<k\le \nu\vee \sigma\}} \, (M_{k}-M_{k-1}),\quad 0\le n\le N.
$$
Since $w$ is previsible, 
$\{ \mu<k\le \nu\vee \sigma\}\in \CF ^ M_k$, and the process $H$ is a martingale. Moreover, 
on $\{\mu=\infty\}=\{ M _N ^ \ast\le \lambda\}$, $S_{N,p}(H)=0$.
On $\{0<\sigma<\infty\}$, the assumption on $w$ leads to
\DEQSZ\nonumber
S ^ p_{N,p}(H) \le S_{\sigma,p} ^ p (H)\le S_{\sigma,p} ^ p (M)= S_{\sigma-1,p} ^ p(M) + (M_\sigma-M_{\sigma-1}) ^ p 
\\\label{here-what}
\le  S_{\sigma-1,p} ^ p(M) +w_\sigma ^ p \le 2\delta ^ p\lambda ^ p.
\EEQSZ
This inequality can also be extended to the case where $\sigma=\infty$. Therefore,
since $E$ is a Banach space of martingale type $p$, by Definition \eqref{def-mart}, there exists a constant $L_p(E)$ such that 
\DEQSZ\label{here-what-2}
\EE |H_N | ^ p \le L_p(E)\, \EE \sum_{k=0}^N |h_k| ^ p \le  L_p(E)\, \EE S_{N,p} ^ p(H) .
\EEQSZ
Substituting \eqref{here-what}, we get
$$
\EE |H_N | ^ p \le 2\, L_p(E)\, \delta  ^ p\lambda ^ p \PP\lk( M ^ \ast_N  >\lambda\rk).
$$
Since $\{ M_N  ^ \ast >\beta \lambda,\, S_{N,p}(M)\vee w_N  ^ \ast \le \delta\lambda\}
\subset \{ H_N ^ \ast >\beta \lambda-\lambda -\delta\lambda\}$, 
\DEQS
\PP\lk( M _N ^ \ast >\beta\lambda,\, S_{N,p}(M) \vee w ^ \ast _N \le \delta \lambda\rk)
&\le& \PP\lk( H_N ^ \ast >\beta \lambda-\lambda -\delta\lambda\rk).
\EEQS
The simple version of Doob's maximal inequality, 
i.e.\ \eqref{doob-simple}, and  \eqref{here-what-2} 
give
\DEQS
\PP\lk( M _N ^ \ast >\beta\lambda,\, S_{N,p}(M) \vee w_N  ^ \ast \le \delta \lambda\rk)
&\le& \PP\lk( H_N ^ \ast >\beta \lambda-\lambda -\delta\lambda\rk)
\\
\le  {1\over (\beta -\delta-1) ^ p \lambda ^ p}  \EE|H_N | ^ p
& \le & 2\, L_p(E)\,{  \delta  ^ p\over (\beta -\delta-1) ^ p } \PP\lk( M_N  ^ \ast >\lambda\rk).
\EEQS
Since $\delta$ can be chosen arbitrary small, the assumptions of Lemma \ref{ausburkholder} are satisfied and
we apply the Lemma to verify the assertion.  The exact constant can be verified by using inequality \eqref{prop-phi-2} and
Definition \eqref{c-phi}.
\end{proof}

\begin{proof}[Proof  of Theorem \ref{burkholder-ineq}]
The proof works in analogy to the proof of Davis, Burkholder and Gundy 
(see e.g.\ \cite[Theorem 1.1]{bdg}, or \cite[Theorem 15.1]{burkholder}).
For an $E$-valued martingale  $\{M_n\}_{n=0}^N  $ and its sequence of martingale differences
$\{m_n\}_{n=0}^N $, let $\CA_n:= \{|m_n|\le 2 m_{n-1} ^ \ast\}$, $0\le n\le N$. 
Davis introduced in \cite{davis} the following decomposition of
$M$, where $
M=G+H,
$
and $G=\{G_n\}_{n=0}^N$ and  $H=\{H_n\}_{n=0}^N   $ are defined by  
\DEQS
G_n &=& \sum_{k=1} ^ n g_k := \sum_{k=1} ^ n\lk[ y_k-\EE[y_k\mid \CF_{k-1}]\rk] ,\quad 0\le n\le N,
\\
H_n  &=& \sum_{k=1} ^ n h_k := \sum_{k=1} ^ n \lk[ z_k+\EE[y_k\mid \CF_{k-1}]\rk] ,\quad  0\le n\le N, 
\EEQS
with $y_k = m_k 1_{\CA_k}$ and $z_k = m_k 1_{\CA ^ C_k}$, $ 0\le k\le N$.
\\
\noindent
Now, since
$$
M ^ \ast_n  \le G _n ^ \ast + H ^ \ast_n , 
$$
by \eqref{prop-phi}, there exists a constant $c_\Phi<\infty$, depending only on $\Phi$, such that
\DEQSZ\label{inqqq-3}
\EE\Phi( M _n ^ \ast ) \le c_\Phi \, \EE \Phi\big( G_n  ^ \ast \big) + c_\Phi\, \EE \Phi \big(   H_n  ^ \ast \big).
\EEQSZ
\noindent

First, we will investigate the last term, i.e.\ $\EE \Phi(H ^ \ast)$ and then we will investigate $\EE \Phi(G ^ \ast)$.
Observe that, since $\{m_n\}_{n=0} ^N$ is a sequence of martingale differences,
$$
\EE\lk[y_k\mid \CF_{k-1}]\rk] +\EE\lk[z_k\mid \CF_{k-1}]\rk]=\EE\lk[m_k\mid \CF_{k-1}]\rk]=0,\quad 1\le k\le N.
$$ 
Therefore, $H_n = \sum_{k=0} ^ n \lk[ z_k-\EE[z_k\mid \CF_{k-1}]\rk]$, $1\le n\le N $.
Applying  the Jensen inequality and  Lemma \ref{lemma-convex} we get
\DEQS
\EE\Phi\big( H_n ^ \ast \big)  &\le&  \EE\Phi\Big(  \sum_{k=0} ^ n \lk|  z_k-\EE[z_k\mid \CF_{k-1}]  \rk|  \Big) 
\\ &\le & 
\EE\Phi\Big(  \sum_{k=0} ^ n \lk|  z_k\rk| +  \sum_{k=0} ^ n \EE[ \lk| z_k\rk| \mid \CF_{k-1}]   \Big) \le
2\, \EE \Phi\Big( \sum_{k=0} ^ n |z_k|\Big).
\EEQS
Since
$|m_k|>2m_{k-1} ^ \ast$  implies $|m_k|<2(m ^ \ast_k-m ^ \ast_{k-1})$, 
and, hence, $|z_k|\le 2(m ^ \ast_k-m ^ \ast_{k-1})$.
Moreover, since $m_n ^ \ast = \sup_{k\le n} |m_k|=\sum_{k=1}  ^ n (m_k ^ \ast-m_{k-1} ^ \ast)$, it follows that
$\sum_{k=1} ^ n|z_k|\le 2 m _n^ \ast$. Therefore,
$$\EE\Phi( H ^ \ast _n ) \le 4\, \EE \Phi( m ^ \ast_n ),\quad 1\le n\le N.
$$
Since $\{|m_n |\}_{n=1} ^ N$ is a non negative real valued sub-martingale, we get by a generalization of
Doob's maximal inequality, i.e.\ Proposition \ref{doob}
\DEQS
\EE\Phi( H^\ast_n ) \le C\, 2\, \EE \Phi(| m_n|  ) .
\EEQS
From $m_n = M_n-M_{n-1}$, it follows that $|m_n|\le S_{n,p}(M)$, and, hence,
\DEQSZ\label{letzte}
\EE\Phi( H^\ast_n ) \le C'\, 2\, \EE \Phi( S_{n,p}( M)  ),\quad 1\le n\le N.
\EEQSZ 
\\[0.3cm]
\noindent

In the next paragraph, we will give an upper estimate of the term $\EE \Phi( G _n ^ \ast )$.
Observe, first, that  $ |m_k|\le  2m_{k-1} ^ \ast$ 
implies $|g_k| \le 4m_{k-1} ^ \ast$, $ 0\le k\le N$.
This means, that $g_k$ is controlled by a $\CF_{k-1}$-measurable random variable,
and, therefore, we can apply   Proposition \ref{inqqq} to get a  control  of $G ^ \ast$.
In particular, there exists a constant $c_{\Phi,E}<\infty$, only depending on $\Phi$ and $E$, such that
\DEQSZ\label{inqqq-2}
\EE \Phi( G ^ \ast _n ) \le c_{\Phi,E}\EE \Phi( S_{n,p}(G)    )  +  c_{\Phi,E}  \EE \Phi(  m _n ^ \ast).
\EEQSZ
The term $ \EE \Phi(  m_n  ^ \ast)$ can be estimated by the generalized Doob maximal inequality \ref{doob}.
So, we obtain 
\DEQSZ\label{inqqq-ll}
\EE \Phi( G ^ \ast_n  ) \le c'_{\Phi,E}\EE \Phi( S_{n,p}(G)     )  +  c'_{\Phi,E}  \EE \Phi(  m _n ).
\EEQSZ
From $m_n = M_n-M_{n-1}$, it follows that $|m_n|\le S_{n,p}(M)$. \del{we get by the tower property
\DEQSZ\label{absss}\nonumber
 \EE\Phi(| m_n|  ) &\le &c_\Phi\EE \Phi(|M_n|)+c_\Phi\EE \Phi(|M_{n-1}|)
\\ &\le& c_\Phi\EE \Phi(|M_n|)+c_\Phi\EE \Phi(|\EE[M_{n}\mid \CF_{n-1}]|)\nonumber
\\ &\le& c_\Phi\EE \Phi(|M_n|)+c_\Phi\EE \Phi( \EE[ |M_{n} |\mid \CF_{n-1} ] )\le 2\,  c_\Phi\EE \Phi(|M_n|).
\EEQSZ
}
It remains to investigate $\EE \Phi( S_{n,p}(G)     )$.
Note, that
\DEQS
\EE\Phi( S_{n,p}(G) )  & \le&  c_{E,\Phi}'' \,  \EE \Phi( S_{n,p}(M)  )   +  c_{E,\Phi}''  \,
\EE \Phi(S_{n,p}(H) ).
\EEQS
Now, since \DEQS 
 \sum_{k\le n} |z_k|^p
&\le&   \sum_{k\le n} |m^\ast_k-m^\ast_{k-1}|^p
\\ &\le&   |m^\ast_n|^{p-1}\sum_{k\le n} |m^\ast_k-m^\ast_{k-1}|\le   |m^\ast_n|^p,
\EEQS
we have by Lemma \ref{lemma-convex} applied to $\{|z_n|^p\}_{n=1}^N$
$$ \EE \Phi(S_{n,p}(H) )\le C \,\EE \Phi\big( m_n ^\ast\big),\quad 1\le n\le N.
$$
Again applying Proposition \ref{doob} leads to
\DEQSZ\label{llllll}
\EE\Phi( S_{n,p}(G) )  & \le&  c''_{E,\Phi} \,  \EE \Phi( S_{n,p}(M)  )+ c''' _{E,\Phi} \EE \Phi\big( |m_n |\big)
\le c''''_{E,\Phi} \,  \EE \Phi( S_{n,p}(M)  ).
\EEQSZ
\del{ \\ &\le&  c'_{E,\Phi}\,   \EE \Phi(  S_{n,p}(M)   )   +  c'_{E,\Phi} \,\EE\Phi\big(H_n ^ \ast \big) ,\quad 1\le n\le N ,
\EEQS}
Therefore, substituting \eqref{letzte}, \eqref{inqqq-ll} and \eqref{llllll} in \eqref{inqqq-3} 
 gives 
\DEQSZ\label{inqqq-4}
\EE\Phi( M ^ \ast _n) \le  \bar c_{E,\Phi}    \EE \Phi(S_{n,p}(M)   )   ,\quad 1\le n\le N.
\EEQSZ
\end{proof}

\del{\subsection{the left hand side}

\begin{theorem}\label{burkholder-ineq-lh}
Let $\Phi:[0,\infty)\to \RR$ be a non decreasing, convex and continuous function with $\Phi(0)=0$ 
and satisfying the growth condition. 
\del{such that
$\Phi$ satisfies the following growth condition: there exists a constants $C>0$ with
$$\Phi( 2\lambda)\le c \Phi(\lambda),\quad \lambda\in [0,\infty).
$$}
\\
Assume also that $2\le q\le \infty $ is fixed and $E$ is a Banach space $E$ of  martingale cotype $q$. Then, 
there exists a constant $C_{q}(E,\phi)>0$ such that for all
$E$-valued  finite martingale $\{M_{n}\}_{n=0}^N$  the
following inequality holds
\DEQSZ 
\nonumber
\mathbb{E}\Phi\lk( \lk(\sum_{n=0}^N | M_{n}-M_{n-1} |_E ^q \rk)^\frac 1p \rk)\le  C_{q}(E,\phi)
 \mathbb{E}\Phi\lk( \sup_{0\le n\le N } | M_{n} |_E \rk)  
\, ,
\label{eqn-2.2-1-lh}
\EEQSZ 
 where as  usually, we put  $M_{-1}=0$.
\end{theorem}

\begin{proposition}\label{inqqq-2-lh}
There exists a constant $C<\infty$ such that
for all $E$-valued martingales  $\{M_n\}_{n=0} ^N $ and  all $M$-previsible processes $\{w_m\}_{m=0}^N $  
satisfying 
$|M_n-M_{n-1}|\le w_n$
for all $1\le n\le N $,
 we have
\DEQS
\EE\Phi( S_{n,q}(M) ) \le C\EE \phi(  M_n ^ \ast )+ C \EE \phi (w_n ^ \ast),\quad 1\le n\le N.
\EEQS
An estimate of the constant is given by $C_{\Phi,p} (E)= \min\{ 2 \delta ^ {-c_\Phi ^ \ast} \beta  ^ {c_\Phi ^ \ast}:\, 
\beta>1,0<\delta<1-\beta$ such that $2L_p(E){\delta ^ p\beta  ^ {c_\Phi ^ \ast} \over (\beta-\delta-1) ^ p}=\frac 12\}$.
\end{proposition}

\begin{proof}[Proof of Proposition  \ref{inqqq-2-lh}]
The proof follows the line of the proof before, only the r\^ole of $S_{N,q}(M)$ and $M^\ast_N$ have been swapped.
Similarly, we will show that the random variables $S_{N,q}(M)$ and $M_N ^ \ast \vee w_N  ^ \ast$ satisfies the assumption of 
Lemma \ref{ausburkholder}.
I.e. we will show that for $\beta>1$ and $0<\delta<\beta-1$ the following holds
\DEQSZ\label{kkk-lh} 
\\
\nonumber
\PP\lk(   S_{N,q}(M)  > \beta \lambda, \,M ^ \ast_N  \vee w ^ \ast_N  \le \delta\lambda\rk) \le
 2\, L_p(X)\,{  \delta  ^ p\over (\beta -\delta-1) ^ p } 
\,\PP\lk( S_{N,q}(M)   >\lambda\rk),\quad \lambda>0.
\EEQSZ
To prove \eqref{kkk-lh}, we introduce again following stopping times, but in comparison to before, slightly modified. Let
$
\mu := \inf\{1\le n\le N :  S_{N,q}(M)  >\lambda\}$, $\nu := \inf\{1\le n\le N:  S_{N,q}(M)  >\beta\lambda\}$, and  
$\sigma := \inf\{1\le n\le N: |M_N|>\delta\lambda\mbox{ or } w_{n}>\delta\lambda\}$.
If the infimum will not be attained, we set the stopping time to $\infty$.
Let $H=\{ H_n,\, 1\le n\le N\}$ be defined by
$$H_n := \sum_{k\le n} h_k=\sum_{k\le n} 1_{\{ \mu<k\le \nu\vee \sigma\}} \, (M_{k}-M_{k-1}),\quad 0\le n\le N.
$$
Since $w$ is previsible, 
$\{ \mu<k\le \nu\vee \sigma\}\in \CF ^ M_k$, and the process $H$ is a martingale. Moreover, 
on $\{\mu=\infty\}=\{ S_{N,q}(M)  \le \lambda\}$, $H_N =0$.
On $\{0<\sigma<\infty\}$, the assumption on $w$ leads to
\del{\DEQSZ\nonumber
M_N  ^ p_{N,p}(H) \le S_{\sigma,p} ^ p (H)\le S_{\sigma,p} ^ p (M)= S_{\sigma-1,p} ^ p + (M_\sigma-M_{\sigma-1}) ^ p 
\\\label{here-what-lh}
\le  S_{\sigma-1,p} ^ p +w_\sigma ^ p \le 2\delta ^ p\lambda ^ p.
\EEQSZ}
\DEQSZ
|H_N | \le  |M_{\sigma}| \le 
|M_{\sigma-1} | + |M_\sigma-M_{\sigma-1}|  
\label{here-what-lh}
\le | M _{\sigma-1} | +w_\sigma  \le 2\delta \lambda .
\EEQSZ
This inequality can also be extended to the case where $\sigma=\infty$. Therefore,
since $E$ is a Banach space of {martingale cotype $q$, } by Definition \eqref{def-mart-ct}, 
there exists a constant $G_p(E)$ and a $0<r<\infty$ such that 
\DEQSZ\label{here-what-2-lh}
\EE S^r_{N,q}(H) \le G_q(E)\, |H_N|^r . 
\EEQSZ
Substituting \eqref{here-what-lh} in \eqref{here-what-2-lh}, we get
$$
\EE S^r_{N,q}(H)  \le G_q(E)\,  2\delta^r \lambda ^r  \PP\lk( S_{N,q}(M)  >\lambda\rk).
$$
Since $\{ S_{N,q}(M)  >\beta \lambda,\, |M_N |\vee w_N  ^ \ast \le \delta\lambda\}
\subset \{ S_{q,N}(H) >\beta \lambda-\lambda -\delta\lambda\}$, 
\DEQS
\PP\lk( S_{N,q}(M) >\beta\lambda,\, |M_N | \vee w ^ \ast _N \le \delta \lambda\rk)
&\le& \PP\lk( S_{q,N}(H) >\beta \lambda-\lambda -\delta\lambda\rk).
\EEQS
The simple version of Doob's maximal inequality, 
i.e.\ \eqref{doob-simple}, and  \eqref{here-what-2-lh} 
give
\DEQS
\PP\lk( S_{N,q}(M) >\beta\lambda,\, |M_N | \vee w ^ \ast _N \le \delta \lambda\rk)
&\le& \PP\lk(  S_{q,N}(H)  >\beta \lambda-\lambda -\delta\lambda\rk)
\\
\le  {1\over (\beta -\delta-1)^r  \lambda^r }  \EE| S^r_{q,N}(H) | 
& \le & 2\, G_q(E)\,  2{  \delta^r  \over (\beta -\delta-1)^r  } \PP\lk( S_{N,q}(M)  >\lambda\rk)  .
\EEQS
Since $\delta$ can be chosen arbitrary small, the assumptions of Lemma \ref{ausburkholder} are satisfied and
we apply the Lemma to verify the assertion.  The exact constant can be verified by using inequality \eqref{prop-phi-2} and
Definition \eqref{c-phi}.
\end{proof}

\begin{proof}[Proof  of Theorem \ref{burkholder-ineq-lh}]
The proof works in analogy to the proof of Davis, Burkholder and Gundy 
(see e.g.\ \cite[Theorem 1.1]{bdg}, or \cite[Theorem 15.1]{burkholder}).
For an $E$-valued martingale  $\{M_n\}_{n=0}^N  $ and its sequence of martingale differences
$\{m_n\}_{n=0}^N $, let $\CA_n:= \{|m_n|\le 2 m_{n-1} ^ \ast\}$, $0\le n\le N$. 
Davis introduced in \cite{davis} the following decomposition of
$M$, where $
M=G+H,
$
and $G=\{G_n\}_{n=0}^N$ and  $H=\{H_n\}_{n=0}^N   $ are defined by  
\DEQS
G_n &=& \sum_{k=1} ^ n g_k := \sum_{k=1} ^ n\lk[ y_k-\EE[y_k\mid \CF_{k-1}]\rk] ,\quad 0\le n\le N,
\\
H_n  &=& \sum_{k=1} ^ n h_k := \sum_{k=1} ^ n \lk[ z_k+\EE[y_k\mid \CF_{k-1}]\rk] ,\quad  0\le n\le N, 
\EEQS
with $y_k = m_k 1_{\CA_k}$ and $z_k = m_k 1_{\CA ^ C_k}$, $ 0\le k\le N$.
\\
\noindent
Now, since
$$
M ^ \ast_n  \le G _n ^ \ast + H ^ \ast_n , 
$$
by \eqref{prop-phi}, there exists a constant $c_\Phi<\infty$, depending only on $\Phi$, such that
\DEQSZ\label{inqqq-3-lh}
\EE\Phi( M _n ^ \ast ) \le c_\Phi \, \EE \Phi\big( G_n  ^ \ast \big) + c_\Phi\, \EE \Phi \big(   H_n  ^ \ast \big).
\EEQSZ
\noindent

First, we will investigate the last term, i.e.\ $\EE \Phi(H ^ \ast)$ and then we will investigate $\EE \Phi(G ^ \ast)$.
Observe that, since $\{m_n\}_{n=0} ^N$ is a sequence of martingale differences,
$$
\EE\lk[y_k\mid \CF_{k-1}]\rk] +\EE\lk[z_k\mid \CF_{k-1}]\rk]=\EE\lk[m_k\mid \CF_{k-1}]\rk]=0,\quad 1\le k\le N.
$$ 
Therefore, $H_n = \sum_{k=0} ^ n \lk[ z_k-\EE[z_k\mid \CF_{k-1}]\rk]$, $1\le n\le N $.
Applying  the Jensen inequality and  Lemma \ref{lemma-convex} we get
\DEQS
\EE\Phi\big( H_n ^ \ast \big)  &\le&  \EE\Phi\Big(  \sum_{k=0} ^ n \lk|  z_k-\EE[z_k\mid \CF_{k-1}]  \rk|  \Big) 
\\ &\le & 
\EE\Phi\Big(  \sum_{k=0} ^ n \lk|  z_k\rk| +  \sum_{k=0} ^ n \EE[ \lk| z_k\rk| \mid \CF_{k-1}]   \Big) \le
2\, \EE \Phi\Big( \sum_{k=0} ^ n |z_k|\Big).
\EEQS
Since
$|m_k|>2m_{k-1} ^ \ast$  implies $|m_k|<2(m ^ \ast_k-m ^ \ast_{k-1})$, 
and, hence, $|z_k|\le 2(m ^ \ast_k-m ^ \ast_{k-1})$.
Moreover, since $m_n ^ \ast = \sup_{k\le n} |m_k|=\sum_{k=1}  ^ n (m_k ^ \ast-m_{k-1} ^ \ast)$, it follows that
$\sum_{k=1} ^ n|z_k|\le 2 m _n^ \ast$. Therefore,
$$\EE\Phi( H ^ \ast _n ) \le 4\, \EE \Phi( m ^ \ast_n ),\quad 1\le n\le N.
$$
Since $\{|m_n |\}_{n=1} ^ N$ is a non negative real valued sub-martingale, we get by a generalization of
Doob's maximal inequality, i.e.\ Proposition \ref{doob}
\DEQS
\EE\Phi( H^\ast_n ) \le C\, 2\, \EE \Phi(| m_n|  ) .
\EEQS
From $m_n = M_n-M_{n-1}$, it follows that $|m_n|\le S_{n,p}(M)$, and, hence,
\DEQSZ\label{letzte-lh}
\EE\Phi( H^\ast_n ) \le C'\, 2\, \EE \Phi( S_{n,p}( M)  ),\quad 1\le n\le N.
\EEQSZ 
\\[0.3cm]
\noindent

In the next paragraph, we will give an upper estimate of the term $\EE \Phi( G _n ^ \ast )$.
Observe, first, that  $ |m_k|\le  2m_{k-1} ^ \ast$ 
implies $|g_k| \le 4m_{k-1} ^ \ast$, $ 0\le k\le N$.
This means, that $g_k$ is controlled by a $\CF_{k-1}$-measurable random variable,
and, therefore, we can apply   Proposition \ref{inqqq} to get a  control  of $G ^ \ast$.
In particular, there exists a constant $c_{\Phi,E}<\infty$, only depending on $\Phi$ and $E$, such that
\DEQSZ\label{inqqqq-2-lh}
\EE \Phi( G ^ \ast _n ) \le c_{\Phi,E}\EE \Phi( S_{n,p}(G)    )  +  c_{\Phi,E}  \EE \Phi(  m _n ^ \ast).
\EEQSZ
The term $ \EE \Phi(  m_n  ^ \ast)$ can be estimated by the generalized Doob maximal inequality \ref{doob}.
So, we obtain 
\DEQSZ\label{inqqq-ll-lh}
\EE \Phi( G ^ \ast_n  ) \le c'_{\Phi,E}\EE \Phi( S_{n,p}(G)     )  +  c'_{\Phi,E}  \EE \Phi(  m _n ).
\EEQSZ
From $m_n = M_n-M_{n-1}$, it follows that $|m_n|\le S_{n,p}(M)$. \del{we get by the tower property
\DEQSZ\label{absss-lh}\nonumber
 \EE\Phi(| m_n|  ) &\le &c_\Phi\EE \Phi(|M_n|)+c_\Phi\EE \Phi(|M_{n-1}|)
\\ &\le& c_\Phi\EE \Phi(|M_n|)+c_\Phi\EE \Phi(|\EE[M_{n}\mid \CF_{n-1}]|)\nonumber
\\ &\le& c_\Phi\EE \Phi(|M_n|)+c_\Phi\EE \Phi( \EE[ |M_{n} |\mid \CF_{n-1} ] )\le 2\,  c_\Phi\EE \Phi(|M_n|).
\EEQSZ
}
It remains to investigate $\EE \Phi( S_{n,p}(G)     )$.
Note, that
\DEQS
\EE\Phi( S_{n,p}(G) )  & \le&  c_{E,\Phi}'' \,  \EE \Phi( S_{n,p}(M)  )   +  c_{E,\Phi}''  \,
\EE \Phi(S_{n,p}(H) ).
\EEQS
Now, since \DEQS 
 \sum_{k\le n} |z_k|^p
&\le&   \sum_{k\le n} |m^\ast_k-m^\ast_{k-1}|^p
\\ &\le&   |m^\ast_n|^{p-1}\sum_{k\le n} |m^\ast_k-m^\ast_{k-1}|\le   |m^\ast_n|^p,
\EEQS
we have by Lemma \ref{lemma-convex} applied to $\{|z_n|^p\}_{n=1}^N$
$$ \EE \Phi(S_{n,p}(H) )\le C \,\EE \Phi\big( m_n ^\ast\big),\quad 1\le n\le N.
$$
Again applying Proposition \ref{doob} leads to
\DEQSZ\label{llllll-lh}
\EE\Phi( S_{n,p}(G) )  & \le&  c''_{E,\Phi} \,  \EE \Phi( S_{n,p}(M)  )+ c''' _{E,\Phi} \EE \Phi\big( |m_n |\big)
\le c''''_{E,\Phi} \,  \EE \Phi( S_{n,p}(M)  ).
\EEQSZ
\del{ \\ &\le&  c'_{E,\Phi}\,   \EE \Phi(  S_{n,p}(M)   )   +  c'_{E,\Phi} \,\EE\Phi\big(H_n ^ \ast \big) ,\quad 1\le n\le N ,
\EEQS}
Therefore, substituting \eqref{letzte-lh}, \eqref{inqqq-ll-lh} and \eqref{llllll-lh} in \eqref{inqqq-3-lh} 
 gives 
\DEQSZ\label{inqqq-4-lh}
\EE\Phi( M ^ \ast _n) \le  \bar c_{E,\Phi}    \EE \Phi(S_{n,p}(M)   )   ,\quad 1\le n\le N.
\EEQSZ
\del{\\[0.3cm]
Finally, using the convexity of $\Phi$ by applying Lemma \ref{lemma-convex}, we can control the
last term on the RHS of \eqref{inqqq-4}. To be more precise, \eqref{prop-phi} gives

\DEQS
 \EE \phi \Big(  \sum_{k=0} ^ n |h_k|\Big) &\le &\, \EE \Phi\lk( \sum_{k=0} ^ n | z_k+\EE [y_k\mid \CF^M_{k-1} ] | \rk)
\\ 
& \le & c_\Phi \,\EE \Phi\lk( \sum_{k=0} ^ n | z_k| \rk)+ c_\Phi \, \EE \Phi\lk( \sum_{k=0} ^ n | \EE [y_k\mid \CF^M_{k-1}] | \rk).
\EEQS}
\end{proof}

\section{comparison between small and large $s$}

\cite{burkholdergundy} p.275, 276, Theorem 5.3 !!

\section{What happens with $0<p<1$}

\cite{burkholdergundy} p. 257
\begin{lemma}
Let $0<p\le 2$. Then there exists a constant $C$ such that we have for any
martingale $\{M_n|\}_{n=0}^N$,
$$
\EE\sup_{0\le n\le N} |M_n|^p \le C\, \EE\lk( \sum_{n=0}^N \EE[ |M_n-M_{n-1}|^2\CF_{n-1} ]\rk)^\frac p2.
$$
\end{lemma}

Laurent schwartz   p radonifying asteriques.

\section{Up and Down-crossing of Martingales}

\cite{burkholderLN} p. 23
On the number of escapes of a martingale

\section{Shrap exponential inequality}

\cite{burkholderLN} p. 35
On the number of escapes of a martingale

}

\del{
\section{Banach spaces and Martingales}

In this section, first, we summarize same basic results of type  and cotype properties of Banach spaces.
Then we define the Banach space which are unconditional for martingale differences
(short UMD) and finally we point out Banach spaces of martingale type.
Moreover, throughout this chapter $B$ denotes a Banach space and $(\Omega,\CF,\PP)$ be any probability space.

\subsection{Banach spaces of type $p$ and cotype $q$}

Here and hereafter $\{\ep_i\}_{i=0}^\infty$ denotes a sequence of independent real valued random variables
with $\PP(\ep_i=1)=\PP(\ep_i=-1)=\frac 12$.

\del{If $\{ x_i\}\subset E$, then we say that the sum $\sum_i \ep_i x_i$ exists, iff there exists a $r>0$ such that 
$$\EE \lk(\sum_i \ep_i x_i\rk)^r<\infty.
$$}
There exists two notations of type,  stable type and Rademacher type. Since
they are nearly equivalent, and coincides in case $p=2$,
we  restrict ourselves to the Rademacher type and
call this property just {\em type}.

The notation of type and cotype Banach spaces were introduced by e.g.\ Hofmann-Jorgensen and Pisier, Maurey and Pisier and Pisier.
Hofmann-Jorgensen and Pisier introduced the notation to investigate 

\begin{definition}
A Banach space $B$ is of {\em Rademacher type $p$} (R-type $p$), \del{ iff $\{x_i\}_{i=0}^\infty \in l^p(E)$ implies
that $\sum_i \ep_i x_i$ exists. 
In other words,} iff there exists a constant $C$ such that for all finite sequences 
$\{x_i\}_{i=0}^N$ in $B$ 
\DEQSZ \label{eq:typ} \EE \Big| \sum_{i=0}^N \ep_i x_i\Big| \le C\, \lk( \sum_{i=0}^N |x_i|^p \rk)^\frac 1p.
\EEQSZ
We denote by $T_p(B)$ the smallest constant for which \eqref{eq:typ} exists.

Conversely, $B$ is of {\em Rademacher cotype $q$}  (R-cotype $q$), iff
there exists a constant $C$ such that for all finite sequences 
$\{x_i\}_{i=0}^N$ in $B$ and a number $r>0$  
\DEQSZ \label{eq:cotyp} \lk( \sum_{i=0}^N |x_i|^q \rk)^\frac 1q \le C\,  \EE \Big| \sum_{i=0}^N \ep_i x_i\Big|.
\EEQSZ
We denote by $C_p(B)$ the smallest constant for which \eqref{eq:typ} exists.
\end{definition}
\begin{remark}
If $p<2$, then if $B$ is of type $p$, then it is of stable type $q$ for $q<p$.
\end{remark}
This properties are quite natural properties, e.g. for  $1\le p\le 2$ any $L^p$-space is of type $p$ (Schwartz 
\cite[Theorem 11.3-11.5]{schwartzLNM}) and cotype $2$,
and for $2\le q<\infty$ any  $L^q$ space is
of type  $2$ and  cotype $q$. 
Moreover, the Minkowski inequality gives that every normed space $X$ is of type $1$ with $T_1(X)=1$.
Clearly every normed space $X$ is of cotype $\infty$, also with $C_\infty=1$.
On the other hand it follows by the Khinthin inequality that no non-zero normed space has type $p$ whenever $p>2$.
Type $p$ implies type $r$ for $r\le p$ and cotype $q$ implies cotype $r$ for $r\ge q$.
A Banach space is of type $2$ and cotype $2$ iff it is isomorphic to a Hilbert space 
(\cite{kwapien72})
When $B$ is of type $p$, then the dual is of cotype $q$ for the conjugate exponent (\cite[p. 1320, Chapter 6]{maurey:03}, or HJ74 (preprint before) and Maurey).
Note, that the converse is not true. the space $l_1$ has cotype $2$, but $c_0$, its dual,
has no non-trivial type. 
In fact, the converse is true, if $B$  is K-convex, a property which was introduced by \cite{maureyPisier76} (Pisier1206, p.206).
Later on, Pisier \cite{pisier82}
has shown  that if $B^\ast$ has cotype $q$ and non trivial type, then, $B$ has type $p$, where $p$ is 
conjugate to $q$.

Note that the norm with  respect to the Rademacher sequence in irrelevant. This follows by the Kahane
inequality, which states that for every $0<q<\infty$ there exists a constants $K^1_q$ and $K^2_q$ such that
for all finite sequences 
$\{x_i\}_{i=0}^N$ in $B$ 
\DEQSZ \label{eq:kahane}K_q^1\, \lk(  \EE \Big| \sum_{i=0}^N \ep_i x_i\Big|^2\rk)^\frac 12\le
 \lk( \EE  \sum_{i=0}^N \Big|\ep_i x_i\Big|^q \rk)^\frac 1q\le K_q^2\, 
\lk(  \EE \Big| \sum_{i=0}^N \ep_i x_i\Big|^2\rk)^\frac 12.
\EEQSZ
If $B=\RR$, then  the inequality \eqref{eq:kahane} reduces to the Khinthin inequality. 
By means of the Kahane inequality, we can derive the following Proposition.
\del{\begin{proposition}(\cite[Proposition 3.1, for martingale p.221]{pisier-1206}) 
Let $E$ be a Banach space of type $p$ and cotype $q$.
Then there exists constants $c,C$ such that for any 
$E$-valued  finite sequence of independent random variables $\{x_{n}\}_{n=0}^N$ and any $0<r<\infty$ the
following inequality holds
\begin{equation} c  \mathbb{E} \lk( \sum_n | x _n  | ^{q}\rk)  ^ \frac rq
\le  \mathbb{E} | \sum_n  x _n   | ^{r} \le  K_q
 \mathbb{E}  \lk( \sum_{n=0}^N  | x_n  | ^{p}\rk) ^ \frac rp,
\label{eqn-2.1}\end{equation}
 where as  usually, we put  $M_{-1}=0$.
\end{proposition}}

\begin{proposition}(\cite[Proposition 3.5.1]{Linde})
The following holds
\begin{itemize}
\item
The Banach space $E$ has R type $p$ iff for some (each) $r\in(0,\infty)$
there exists a constant $c>0$ such that 
$$
\lk\{ \EE \lk| \sum_{i=0}^n\ep_i x_i\rk| ^r \rk\}^\frac 1r \le \lk\{ \EE  \sum_{i=0}^n\lk| x_i\rk| ^p \rk\}^\frac 1p
$$
for all $x_1,\ldots , x_n\in E$.
\item
The Banach space $E$ has R cotype $q$ iff for some (each) $r\in(0,\infty)$
there exists a constant $c>0$ such that 
$$
 \lk\{ \EE  \sum_{i=0}^n\lk| x_i\rk| ^q \rk\}^\frac 1q\le \lk\{ \EE \lk| \sum_{i=0}^n\ep_i x_i\rk| ^r \rk\}^\frac 1r \le
$$
for all $x_1,\ldots , x_n\in E$.
 \end{itemize}
\end{proposition}

\begin{example}
Suppose $E=L^p(U,\nu)$. Then $E$ is of R type $p$.
that $0\le q<p\le 2$ and let $E$ be a Banach space
of type  
\end{example}

Araujo and Gine \cite{araujo} investigated the relationship between Banach spaces of type $p$ and 
the characterization of L\'evy measures. E.g. they proved that a Banach space is of cotype $2$ iff every L\'evy measure 
$\mu$ on $B$ satisfies (\cite[Theorem 2.2]{araujo})
$$ \int \min( 1,|x|^2) \nu(dx)<\infty.
$$
If $p<2$ only a weaker version holds. A Banach space is of type $p$ iff every L\'evy measure 
$\mu$ on $B$ integrates $  \min( 1,|x|^2) $  and gives mass zero to $\{0\}$ (\cite[Theorem 2.3]{araujo}). 
Relation to the L\'evy khintchin formula also see Linde, Dettweiler.

\subsection{Banach spaces which are unconditional for martingale differences}

By means of the type $p$ and cotype $q$ property, one can define the stochastic integral, but only for 
deterministic integrand. To define a stochastic integral with respect to an integrand, which
is itself random, one have to know that the underlying Banach spaces is unconditional for martingale. 

\begin{definition}
A Banach space $E$ is UMD, iff for any $1<p<\infty$
there exists a constant $\beta_{p}(E)>0$ such that for all
$E$-valued  finite martingale $\{M_{n}\}_{n=0}^N$ and any Rademacher sequence $ \{\ep_{n}\}_{n=0}^N$ the
following inequality holds
\DEQS \mathbb{E} \Big|\sum_{n=0}^N \ep_n( M_{n}-M_{n-1})  \Big|_E ^{p}   
\le  \beta_{p}(E)\,
 \EE\Big|\sum_{n=0}^N  M_{n}-M_{n-1}  \Big|_E ^p ,
\EEQS 
 where as  usually, we put  $M_{-1}=0$.
\end{definition}

This property was of great interest and is closely related to questions investigated by ...
(burkholder, AP 9,1981).
An important fact is that it is equivalent to a convexity property, i.e.
if $B$ is UMD, then it is B convex. For a survey on this topic we refer to Burkholder 1206.
or Burkholder AoP 9,6 \cite{burkholder81} here he proofs Mt UMD

If $B$ is a real or complex Hilbert space, then $\beta_p(B)=p^\ast -1$ for all $p\in(0,\infty)$,
where $p^\ast=p\vee q$, $q$ conjugate to $p$. The spaces $l^1$, $l^\infty$, $L^1(0,1)$ and $L^\infty(0,1)$ are not UMD. In fact, UMD spaces are reflexive
even super reflexive (Remark 4 Maurey 1975), Aldous 1979.

Burkholder proofed that UMD is equivalent to $\xi$ convexity.
AOP 1981

\subsection{Martingale type $p$ and Martingale cotype $p$}

In this section we collect some basic informations about the     martingale type $p$, $p\in [1,2]$, Banach spaces.
For more details we refer to \cite[Appendix C]{maxreg} or \cite{MR1313905}.

A property which encompasses both, the UMD property and the type $p$ property is the martingale type $p$ property.

\begin{definition}\label{def-mart}
Assume that $p\in [1,2]$ is fixed. A Banach space $E$ is of  martingale type $p$  iff
there exists a constant $L_{p}(E)>0$ such that for all
$E$-valued  finite martingale $\{M_{n}\}_{n=0}^N$  the
following inequality holds
\DEQS \sup_{0\le n\le N } \mathbb{E} | M_{n} |_E ^{p} \le  L_{p}(E)\,
\mathbb{E}  \sum_{n=0}^N  | M_{n}-M_{n-1} |_E ^{p},
\EEQS 
 where as  usually, we put  $M_{-1}=0$.
\end{definition}

Similarly one can define the martingale cotype property.

\begin{definition}\label{def-mart-ct}
Assume that $q\in [2,\infty ]$ is fixed. A Banach space $E$ is of  martingale cotype $q$  iff
there exists a constant $C_{p}(E)>0$ such that for all
$E$-valued  finite martingale $\{M_{n}\}_{n=0}^N$  the
following inequality holds
\DEQS 
\mathbb{E}  \sum_{n=0}^N  | M_{n}-M_{n-1} |_E ^{q}\le  C_{q}(E)\,
\sup_{0\le n\le N } \mathbb{E} | M_{n} |_E ^{q}  ,
\EEQS 
 where as  usually, we put  $M_{-1}=0$.
\end{definition}

\begin{remark}
Another terminus which is related to the martingale type property is the uniform smoothness of a Banach space.
This is a purely geometric condition,  and existed before the  martingale type property.
To be more precise, Pisier and Woyzz,,, \cite[Propsotion 2.4, p.\ 336]{Pisier75} have shown that every $p$-uniformly smooth Banach space is of martingale
type $p$.
Note, that there exists Banach spaces which are of type $p$, but not of martingale type $p$.
R. James constructed an example of a Banach space which
is of type $p$ but not of martingale type $p$ uniformly (w... p. 253)

Similar definition can also be found by Pisier, chapter 6, LN 1206 and Wyczynski LN 472, 1975.
Woyzinski has shown, that a Banach space is of martingale type $p$ iff the Banach spaces is $p$-uniformly smooth.
(Theorem 2.2, p. 251)
\end{remark}Preliminares about conv

\subsection{Sharpness of the inequalities}

to go to the limit: Burkholder AOP 15,1

Hilbert spaces the best constant \cite{burkholder2002}
for $p=1$ and $1>p\le 2$ for the quadtratic variation (RHS).
}
\section{Preliminaries about convex functions}
\label{appendix-convex-function}

Let $\Phi:[0,\infty)\to \RR$ be a strictly increasing convex function.
By e.g.\ \cite[Theorem A]{convex-sum} it follows that there exists a function $\phi:\RR_+\to \RR_+$, where $\phi$ is strictly increasing such that
$$\Phi(t) = \int_0^t \phi(s)\, ds, \quad t\ge 0.
$$
To such a convex function $\Phi$ we can associate another convex function $\Psi$
of the same type such that
$$\Psi(t) = \int_0^t \psi(s)\, ds, \quad t\ge 0.
$$
and $\phi(s)=\inf\{ t\ge 0: \psi(t) \ge s\}$ and $\psi(s)=\inf\{ t\ge 0: \phi(t) \ge s\}$, $s\ge 0$.
(see e.g.\ \cite[Chapter I.15]{convex-sum})

Particulary, the following holds (see e.g.\ \cite[Chapter I.15, p.30]{convex-sum})
\begin{proposition}
Let $\Phi:[0,\infty)\to \RR$ be a strictly increasing convex function and $\Psi:[0,\infty)\to \RR$ its conjugate.
Then
\DEQSZ u\,\phi(u) &= &\Phi(u)+\Psi(\phi(u)),\label{fig:01}
\\ \int_0^v t \, d\psi(t) &=& \Phi(\psi(v)),\label{fig:02}
\\ u\, v &\le & \Phi(u) +\Psi(v),\label{fig:03}
\\ \Phi(ua)  &\le& a \Phi(a),\quad \forall 0<a\le 1,\label{fig:04}
\EEQSZ
\end{proposition}

Furthermore,
we say $\Phi$ satisfies {\em the growth condition},
iff  there exists a constant $c_\Phi$ with
$$\Phi( 2\lambda)\le c_\Phi\,  \Phi(\lambda),\quad \lambda\in [0,\infty).
$$
Since we will need it later, we summaries in this paragraph some facts about convex functions (see e.g.\ \cite[Appendix]{kras} or \cite{convex-sum}). 
For any  increasing convex and continuous function $\Phi$,  there exists an increasing, non-negative
function $\phi:[0,\infty)\to [0,\infty)$ with $\phi(0)=0$ such that $\Phi(t) = \int_0^t \phi(s)\, ds$.
We can associate to $\Phi$ a function $\Psi$, where $\Psi(t)= \int_0^t \psi(s)\, ds$
and 
$\psi(s)= \sup\{ t: \phi(t)\le s\}$ for all $s>0$. Such a function is called conjugate to $\Phi$  in the sense of Young.
If the growth condition holds, then 
\DEQSZ\label{c-phi} c_\Phi ^ \ast := \sup_{u>0}{u\phi(u)\over \Phi(u)}
\EEQSZ
is finite and we get 
\DEQSZ\label{prop-phi}
 \Phi( t_1\vee t _2) &\le & \Phi( t_1)+\Phi(t_2), \quad t_1,t_2\ge 0,
\\\Phi( r t ) &\le& r ^ {c_\Phi ^ \ast } \Phi(t), \quad t\ge ,\, r\ge 1,\label{prop-phi-2}
\\
\Psi(t) &\le& (c_\Phi ^ \ast -1)\, \Phi(\psi(t)),\quad t\ge 0
.\label{convex}
\EEQSZ

\del{
\section{References}

Uniform convexity: introduced by Clarkson 1936 in transamer math soc. 40

Examples: Milman: $L ^ p$ spaces $p\ge 1$. uniformly convex 
1971, uspekhi mat. naut 26 -6 

uniform smoothness day 1 : 
}
\def\cprime{$'$} \def\cprime{$'$} \def\cprime{$'$} \def\cprime{$'$}

\end{document}